\documentclass[smallextended,natbib,runningheads]{svjourMein}
\smartqed 

\makeatletter
\def\cl@chapter{\@elt {theorem}}
\makeatother

\usepackage{placeins}
\usepackage{float}
\usepackage{a4wide}
\usepackage{latexsym}
\usepackage[USenglish]{babel}
\usepackage[utf8]{inputenc}
\usepackage[T1]{fontenc}
\usepackage[breaklinks=true]{hyperref}
\usepackage{amssymb}
\usepackage{amsfonts}
\usepackage{amsmath,amssymb}
\usepackage{hhline}
\usepackage{float}
\usepackage{mathrsfs}
\usepackage{graphicx}
\usepackage{amsmath}
\usepackage{mathdots}
\usepackage{amsfonts}
\usepackage{ntheorem}
\usepackage{amsxtra}
\usepackage{url}
\usepackage{mathptmx}
\usepackage{helvet}
\usepackage{stmaryrd}
\usepackage{dsfont}
\usepackage{setspace}
\usepackage{enumerate}
\usepackage{MathDef}
\usepackage{verbatim}
\usepackage{color}
\usepackage{cleveref}
\usepackage{cancel}

\def\C{\mathbb{C}}
\def\E{\mathbb{E}}
\def\N{\mathbb{N}}
\def\P{\mathbb{P}}

\def\R{\mathbb{R}}
\def\Z{\mathbb{Z}}

\def\1{\mathds{1}}

\def\Var{\text{Var}}

\def\Ind{\mathds{1}}

\def\Cov{\text{Cov}}

\def\f{f_Y^{(\Delta)}}
\def\tr{\operatorname{tr}}

\def\pa2{\frac{\partial ^2}{\partial \vartheta ^2}}

\newcommand\D{^{(\Delta)}}
\def\vecc{\operatorname{vec}}
\def\sc{\textsc}
\def\bf{\textbf}

\newtheorem{theorem1}{Theorem}
\newtheorem*{assumptionletterA}{{\textbf{Assumption A}}}
\newtheorem*{assumptionletterB}{{\textbf{Assumption B}}}
\newtheorem*{assumptionletterC}{{\textbf{Assumption $\widetilde{\text{A}}$}}}
\newtheorem*{assumptionletterD}{{\textbf{Assumption $\widetilde{\text{B}}$}}}

\begin{document}

\title{Whittle estimation for stationary state space models with finite second moments 
}


\author{Vicky Fasen-Hartmann      \and
        Celeste Mayer 
}


\institute{Vicky Fasen-Hartmann \at
             Institute of Stochastics, Englerstra{\ss}e 2,
             D-76131 Karlsruhe, Germany \\
              \email{vicky.fasen@kit.edu}           
           \and
           Celeste Mayer \at
             Institute of Stochastics, Englerstra{\ss}e 2,
             D-76131 Karlsruhe, Germany.\\
\email{celeste.mayer@kit.edu}             \\}


\maketitle

\begin{abstract}
In this paper, we consider the Whittle estimator for the parameters of a stationary solution of a continuous-time linear state space model
sampled at low frequencies. In our context the driving process is a L\'evy process which allows flexible margins of the underlying model.
The L\'evy process is supposed to have finite second moments. It is well known that then the class of stationary solutions
of linear state space models and the class of multivariate CARMA processes coincides.
We prove that the Whittle estimator,
which is based on the periodogram,
is strongly consistent and asymptotically normally distributed.
A comparison with the classical setting of discrete-time ARMA models shows that in the continuous-time
setting the limit covariance matrix of the Whittle estimator has an additional correction term for non-Gaussian models.
For the proof, we investigate as well the asymptotic normality of the integrated periodogram which is interesting for its own.
It can be used to construct goodness of fit tests.
Furthermore, for univariate state space processes, which are  CARMA processes, we introduce an adjusted version of the Whittle estimator
and derive as well the asymptotic properties of this estimator.
The practical applicability of our estimators is demonstrated through a simulation study. 
\end{abstract}

\section{Introduction}
 Continuous-time linear state space models
are widely used in diversified fields as, e.g., in signal processing and
control, high-frequency financial econometrics and financial mathematics. The advantages of continuous-time models
are that they allow to model high-frequency data as in finance and in turbulence but as well irregularly spaced data, missing observations
or situations when estimation and inference at various frequencies has to be carried out.

In this paper, we investigate stationary solutions of continuous-time  linear state space models driven by a Lévy process. 	
A one-sided $d$-dimensional Lévy process $(L_t)_{t \geq 0}$ is a stochastic process with stationary and independent increments satisfying $L_0=0$ almost surely
and having continuous in probability sample paths.
For matrices $A\in \R^{N\times N}$, $B\in \R^{N\times d}$, $C\in \R^{m\times N}$ and an
$d$-dimensional centered Lévy process $L=(L_t)_{t\geq 0}$ a continuous-time linear state space model $(A,B,C,L)$ is defined by
\begin{align}
\label{statespace}
\begin{array}{rcl}
dX_t&=& AX_tdt+ BdL_t, \\
Y_t&=& CX_t,   \quad t\geq 0.
\end{array}
\end{align}
The processes $(X_t)_{t\geq 0}$ and $(Y_t)_{t\geq 0}$ in the state space representation
\eqref{statespace} are called state- and output process, respectively.

In the case of a finite second moment of the driving Lévy process the classes of stationary linear state space models and multivariate
continuous-time ARMA (MCARMA) models are equivalent (see \cite{Schlemm:Stelzer:2012}, Corollary 3.4).
This means that for every output process $(Y_t)_{t\geq 0}$ of the state space model \eqref{statespace} there exist
an autoregressive polynomial
$    {P}(z):= I_{d}z^p+P_1z^{p-1}+\ldots+P_{p-1}z+P_p$
with $P_1,\ldots, P_p\in \R^{d\times d}$ and a moving average polynomial
${Q}(z):=Q_0z^q+ Q_1z^{q-1}+\ldots+ Q_{q-1}z+Q_q$
with $Q_0,\ldots,Q_q\in \R^{d\times m}$ such that $(Y_t)_{t\geq 0}$ can be interpreted
as  solution of the differential equation
\begin{eqnarray} \label{MCARMA}
{P}(\mathsf{D})Y_t={Q}(\mathsf{D})\mathsf{D} L_t, \quad  t\geq 0,
\end{eqnarray}
where $\mathsf{D}$ is the differential operator with respect to $t$. Since the orders of the autoregressive  polynomial and the moving average polynomial
are $p$ and $q$, $Y$ is called MCARMA$(p,q)$ process.
Formally,  MCARMA processes were introduced as linear state space models with special  matrices $A, B, C,$ see \cite{Marquardt:Stelzer:2007}. Since the parametrization of a general linear state space model \eqref{statespace} is more flexible than the parametrization of an MCARMA model \eqref{MCARMA}, it is advantageous to use \eqref{statespace} and estimate the parameters within this representation.

The defining differential equation \eqref{MCARMA} of an MCARMA process reminds of the defining difference equation of a discrete-time vector ARMA (VARMA) process. A
VARMA process $(Z_n)_{n\in\N}$ is the $d$-dimensional solution of a difference equation of the form
\begin{eqnarray} \label{VARMA}
{P}(\mathsf{B})Z_n={Q}(\mathsf{B}) e_n, \quad n\in\N,
\end{eqnarray}
where $\mathsf{B}$ is the Backshift-operator $\mathsf{B}Z_n=Z_{n-1}$ and $(e_n)_{n\in\Z}$ is an $m$-dimensional white noise, see, e.g., the monographs
of \cite{BrockwellDavis} and \cite{Lutke}. From \cite{Thornton:Chambers:2017}, see  \cite{Brockwell:Lindner:2009} for the univariate case,
it is well  known that a discretely sampled MCARMA process  admits a VARMA
representation with a weak white noise $(e_n)_{n\in\Z}$. The covariance
matrix of $e_n$ depends on the parameters of the polynomial $P$ and $Q$ in the MCARMA
representation, respectively on the parameters of $(A,B,C)$ in the state space model \eqref{statespace}. For Lévy driven models the white noise of the sampled process is whether
a strong white noise nor a martingale difference in general.
Since the results concerning the asymptotic behavior of the quasi maximum likelihood estimator and the Whittle estimator for VARMA models require the white noise to be a martingale difference, see  \cite{DunsmuirHannan76},  \cite{deistler_dunsmuir_hannan_1978}, \cite{Dahlhaus:Poetscher:1989}, they are not transferable to non-Gaussian Lévy driven state space models.

In the econometric literature there are several papers using the
Kalman filter approach
for maximum likelihood estimation
of Gaussian possibly non-stationary MCARMA processes as, e.g.,
\cite{Harvey:Stock:1985,Harvey:Stock:1988,HarveyStock1989}, \cite{Zadrozny1988}, \cite{Thornton:Chambers:2017}.
The rigorous
mathematical derivation of the asymptotic properties of the quasi-maximum
likelihood estimator for stationary Lévy driven state space and MCARMA models was given
recently in \cite{QMLE} 
and for non-stationary  models in \cite{FasenScholz:2019}. In the case of univariate MCARMA processes with $d=m=1$, which are called CARMA processes,
there exist some further estimation methods.
An indirect estimation procedure for CARMA models, which is robust against outliers, is topic of
\cite{Fasen:Kimmig:2020}. To the best of our knowledge \cite{Fasen:Fuchs:2013b} present the only frequency domain estimator for high-frequency sampled CARMA processes.

In this paper, we investigate a frequency domain estimator, the Whittle estimator, for a low-frequency sampled state space model \eqref{statespace} with stationary observations $Y_{\Delta},\ldots,Y_{n\Delta}$
($\Delta>0$ fixed). 
The Whittle  estimator is going back to \cite{Whittle:1951,whittle1953estimation}, \cite{Walker:1964} and is very well investigated for different time series models
in discrete time.
If the autocovariance function of  $Y^{(\Delta)}:=(Y_{k\Delta})_{k\in\N_0}$ is denoted by
$\Gamma_Y^{(\Delta)}(h)=\operatorname{Cov}(Y_{(h+1)\Delta},Y_{\Delta})$ and $\Gamma_Y^{(\Delta)}(-h)=\Gamma_Y^{(\Delta)}(h)^\top$,
$h\in\N_0$, the spectral density $f_Y^{(\Delta)}$ of $Y^{(\Delta)}$ is defined as Fourier transform of the autocovariance function
\beam \label{spek}
f_Y^{(\Delta)}(\omega)=\frac{1}{2\pi}\sum_{h \in \Z}\Gamma_Y^{(\Delta)}(h)e^{-ih\omega}, \quad \omega \in [-\pi,\pi].
\eeam
Conversely, using the inverse Fourier transform, yields
\begin{align}\label{kovarianz}
\Gamma_Y^{(\Delta)}(h)=\int_{-\pi}^{\pi}f_Y^{(\Delta)}(\omega)e^{ih\omega}d\omega, \quad   h\in\Z.
\end{align}
The empirical version of the spectral density is
the periodogram $I_n:[-\pi,\pi]\rightarrow \R^{m\times m}$ defined as
\begin{align}\label{Per_ACVF}
I_n(\omega)=\frac{1}{2\pi n}\left(\sum_{j=1}^{n}Y_{j\Delta}e^{-ij\omega}\right)\left(\sum_{k=1}^{n}Y_{k\Delta}e^{ik\omega}\right)^\top
=\frac{1}{2\pi}\sum_{h=-n+1}^{n-1}\overline{\Gamma}_n(h)e^{-ih\omega}, \quad  \omega \in [-\pi, \pi],
\end{align}
where
$$\overline{\Gamma}_n(h):=\frac1n\sum_{k=1}^{n-h}Y_{(k+h)\Delta}Y_{k\Delta}^\top \quad \mbox{ and } \quad \overline{\Gamma}_n(-h):=\overline{\Gamma}_n(h)^\top, \quad h=0,\ldots,n,$$
is the empirical autocovariance function.
For different frequencies
the periodogram behaves asymptotically like independent exponentially distributed random
variables, 
see \cite{fasen2013a}, and is not a consistent estimator for the spectral density. However, the periodogram is the basic part of the Whittle estimator.

Let $\Theta\subseteq \R^r$ be a parameter space and for any $\vartheta\in\Theta$ let $f_Y^{(\Delta)}(\omega, \vartheta)$
be the spectral density of a stationary equidistant sampled state space process $Y^{(\Delta)}(\vartheta)$.
Then, the Whittle function $W_n$ is defined by
$$W_n(\vartheta)=\frac{1}{2n}\sum_{j=-n+1}^{n}\Big[\tr\left(f_Y^{(\Delta)}(\omega_j, \vartheta)^{-1}I_n(\omega_j)\right)+\log\left(\det \left(f_Y^{(\Delta)}(\omega_j,\vartheta)\right)\right)\Big], \quad  \vartheta\in\Theta,$$
with
$\omega_j=\frac{\pi j}{n}$ for $ j=-n+1,\ldots,n$
and the  Whittle estimator is
$$\widehat{\vartheta}_n^{(\Delta)}:= \arg \min_{\vartheta\in \Theta} W_n(\vartheta).$$
In the definition of the Whittle function it is also possible to replace the term
$\log(\det (f_Y^{(\Delta)}(\omega_j,\vartheta)))$  by $\log(\det V^{(\Delta)}(\vartheta))$ where $V^{(\Delta)}(\vartheta)$ is the covariance
matrix of the one-step linear prediction error.
Therefore, if the covariance matrix $ V^{(\Delta)}(\vartheta)$ of the linear prediction error does not depend on $\vartheta$, we can neglect the penalty term   $\log(\det V^{(\Delta)}(\vartheta))$ completely since it is constant for all
$\vartheta$.
However, in the case of state space models, $V^{(\Delta)}(\vartheta)$ depends on $\vartheta$ and
has to be computed additionally (cf. \Cref{invertible}). Conversely, for VARMA models,
$V^{(\Delta)}(\vartheta)$ is the covariance matrix of the
white noise.   Hence, the Whittle function for VARMA models with penalty function $\log(\det V^{(\Delta)}(\vartheta))$
in \cite{DunsmuirHannan76} differs from our Whittle function. That paper is also one of the few papers using the
Whittle estimator for the estimation of a multivariate model.


Empirical spectral processes indexed by a class of functions are applied to derive the asymptotic properties
of frequency domain estimators  as the  Whittle estimator.
The asymptotic behavior of empirical spectral processes is very well investigated
but unfortunately the known results  cannot be utilized to our setting. 
The empirical spectral process theory usually requires some exponential inequality and therefore some stronger model assumptions are necessary.
For example, \cite{Mikosch:Norvaisa} investigate empirical spectral processes for linear models with i.i.d. (independent and identically distributed) noise having finite fourth moments; similarly
\cite{dahlhaus2009empirical}. 
\cite{dahlhaus1988empirical} assumes some exponential moment condition for the stationary time series model
and  \cite{Dahlhaus:Polonik:2006} study Gaussian locally stationary processes. The recent paper of  \cite{Bardet:Doukhan:Leon}
assumes some weak dependence on the stationary time series and that the one-step linear prediction error variance, which corresponds
to the variance of the white noise in the ARMA representation of the discrete sampled process, does not depend
on the model parameters. However, in our case, the parameters of $(A,B,C)$ affect this variance.
Whittle estimation for continuous-time fractionally integrated CAR processes, where the driving process is a
fractionally Brownian motion, is studied in  \cite{tsai}. But essential for the proofs in that paper is again
that the driving process is Gaussian such that the techniques cannot be used  for Lévy driven models.
Moreover, all of these papers only analyze  univariate models, whereas we consider a multivariate model.

The paper is structured in the following way. We start by stating the basic facts on discrete-time sampled
linear state space models in \Cref{sec:preliminaries}. Then, the main results of this paper are presented. In \Cref{sec:WhittleEstimator}, we derive the consistency and the asymptotic
normality of the Whittle estimator. Interesting is that for non-Gaussian state space models the limit covariance matrix
of the Whittle estimator differs from the covariance matrix in the Gaussian case. As a contrast to
Whittle estimation for VARMA models, this confirms
that for the proofs standard techniques cannot be applied as well. An advantage of the Whittle estimator over the quasi-maximum likelihood
estimator of \cite{QMLE} is that we have an analytic representation of the limit covariance matrix which can be used for
the determination of confidence bands. For the proof of the asymptotic normality of the Whittle estimator
we show as well the asymptotic normality of the integrated periodogram. This result lays the basis for goodness
of fit tests for state space models which can be written as continuous functionals of the integrated periodogram as, e.g., the Grenander
and Rosenblatt test or Bartlett's
test for the integrated periodogram, Bartlett's $T_p$ test or the Cramér-von Mises test (cf. \cite{Priestley:1981}), and is topic of some
future research.
Furthermore, results of this type are typically used for bootstraps in the frequency domain.
In \Cref{sec:adjustedWhittleestimator}, we
motivate the definition of the adjusted Whittle estimator, which works only for univariate state space models with $d=m=1$, and present
 the consistency and the asymptotic normality for this estimator as well.
Finally, the applicability of the Whittle and the adjusted Whittle estimator is demonstrated through a simulation study
in \Cref{sec:simulation} and compared to the quasi maximum likelihood estimator
of \cite{QMLE}. 
 for the Whittle estimator, the detailed
  proofs are given in \Cref{sec:Proofs:WhittleEstimator} and 
 since the proofs for the adjusted Whittle estimator are very similar, they are moved
 to  \Cref{sec:Proofs:adjustedWhittle} in the Supplementary Material.
Some further simulation studies are presented there as well.


\subsection*{Notation}
For some matrix $A$, $\tr(A)$ stands for the trace of $A$, $\det(A)$ for its determinant, $A^\top$ for its transpose and $A^H$ for the transposed complex conjugated matrix. Further, $A[i,j]$ denotes the $(i,j)$-th component of $A$. We write $\vecc(A)$ for the vectorization of $A$ and $A\otimes B$ for the Kronecker product of $A$ and $B$ where $B$ is any matrix. The $N$-dimensional identity matrix is denoted as $I_N$. For a matrix function $g(\vartheta)$ in $\R^{m\times s}$ with $\vartheta$ in $\R^r$ the gradient with respect to the parameter vector $\vartheta$ is denoted by $\nabla_\vartheta g(\vartheta)=\frac{\partial \vecc(g(\vartheta))}{\partial \vartheta}\in \R^{ms\times r}$ and $\nabla_\vartheta g(\vartheta_0)$ is the shorthand for $\nabla_\vartheta g(\vartheta)|_{\vartheta=\vartheta_0}.$ If $g:\R^{r}\to \R$, then $\nabla_\vartheta ^2 g(\vartheta) \in \R^{r\times r}$ denotes the Hessian matrix of $g(\vartheta).$  For the real and the imaginary part of a complex valued $z$, we use the notation $\Re(z)$ and $\Im(z)$, respectively. Throughout the article, $\| \cdot\|$ denotes an arbitrary sub-multiplicative matrix norm. Finally, $\mathfrak{C}>0$ is a constant which may change from line to line.

\section{Preliminaries} \label{sec:preliminaries}

Let $\Theta \subset \R^{r}$ be a parameter space, and suppose that for any $\vartheta \in \Theta$, $A(\vartheta)\in \R^{N\times N}$ has eigenvalues with strictly negative real parts, $B(\vartheta)\in \R^{N\times d}$, $ C(\vartheta)\in \R^{m\times N}$ and
$L(\vartheta):=(L_t(\vartheta))_{t \in \R}$ is an $\R^d$-valued L\'evy process with existing covariance matrix $\Sigma_L(\vartheta)$.
A two-sided Lévy process can be constructed from two independent one-sided Lévy processes $(L^{(1)}_t(\vartheta))_{t\geq 0}$ and $(L^{(2)}_t(\vartheta))_{t\geq 0}$
through
$L_t(\vartheta)=L_{t}^{(1)}(\vartheta)\mathbf{1}_{ \{t\geq 0\}}-\lim_{s\uparrow -t}L^{(2)}_{s}(\vartheta)\mathbf{1}_{\{t<0\}}.$
Details on Lévy processes can be found in \cite{Sato}.
The stationary solution of the state space model
\begin{align*}
Y_t(\vartheta)= C(\vartheta)X_t(\vartheta) \quad \text{ and } \quad dX_t(\vartheta)= A(\vartheta)X_t(\vartheta)dt+ B(\vartheta)dL_t(\vartheta), \quad t\geq 0,
\end{align*}
has the representation
\beao
Y_t(\vartheta)= C(\vartheta)X_t(\vartheta) \quad \text{ and } \quad X_t(\vartheta)=\int_{-\infty}^{t} \e^{A(\vartheta)(t-s)}B(\vartheta)\,dL_s(\vartheta),  \quad t\geq 0.
\eeao
 The true parameter of the output process $Y$ of our observations $Y_{\Delta},\ldots,Y_{n\Delta}$ is denoted by $\vartheta_0$ and is supposed to be in $\Theta$.
Since we only observe the output process of the state space model at discrete time points with distance $\Delta>0$, we are interested in the probabilistic properties of $Y^{(\Delta)}(\vartheta):=(Y_k^{(\Delta)}(\vartheta))_{k\in \N_0}:=(Y_{k\Delta}(\vartheta))_{k\in \N_0}$ as well.
The discrete-time process $Y^{(\Delta)}(\vartheta)$ has the discrete-time state space representation
\begin{align*}
Y_k^{(\Delta)}(\vartheta)=C(\vartheta)X_k^{(\Delta)}(\vartheta) \quad  \text{ and } \quad
X^{(\Delta)}_k(\vartheta)=e^{A(\vartheta)\Delta}X^{(\Delta)}_{k-1}(\vartheta)+N^{(\Delta)}_k(\vartheta), \quad  k\in\N_0,
\end{align*}
where
$$ N_k^{(\Delta)}(\vartheta)=\int_{(k-1)\Delta}^{k\Delta}e^{A(\vartheta)(k\Delta-u)}B(\vartheta)dL_u(\vartheta), \quad  k\in\N_0, $$
is  an i.i.d. sequence with mean zero and covariance matrix
\beao \label{cov:N}
\Sigma_N^{(\Delta)}(\vartheta)=\int_0^{\Delta}e^{A(\vartheta)u}B(\vartheta)\Sigma_L(\vartheta)B(\vartheta)^\top e^{A(\vartheta)^\top u}du
\eeao
(see \cite{QMLE}, Proposition 3.6).
Furthermore, 
$Y^{(\Delta)}(\vartheta)$ has the vector MA$(\infty)$ representation $$Y_k^{(\Delta)}(\vartheta)=\sum_{j=0}^\infty \Phi_j(\vartheta) N_{k-j}^{(\Delta)}(\vartheta),\quad k\in\N_0,$$ where $\Phi_j(\vartheta)=C(\vartheta)e^{A(\vartheta)\Delta j} \in \R^{m\times N}$.
Defining $\Phi(z,\vartheta):=\sum_{j=0}^{\infty}\Phi_j(\vartheta)z^j,\,   z\in\C,$
an application of  \cite{BrockwellDavis}, Theorem 11.8.3, gives the spectral density
\begin{eqnarray} \label{spec}
\f(\omega,\vartheta)&=&\frac{1}{2\pi}\Phi(e^{-i\omega},\vartheta)\Sigma_N^{(\Delta)}(\vartheta)\Phi(e^{i\omega},\vartheta)^\top \\ &=&\frac{1}{2\pi}C(\vartheta)\left(e^{i\omega}I_{N}-e^{A(\vartheta)\Delta}\right)^{-1}\Sigma_N^{(\Delta)}(\vartheta)\left(e^{-i\omega}I_{N}-e^{A(\vartheta)^\top \Delta}\right)^{-1}C(\vartheta)^\top, \quad \omega \in [-\pi,\pi], \nonumber
\end{eqnarray}
of $Y^{(\Delta)}(\vartheta)$.
For better readability, we will omit the true parameter $\vartheta_0$ whenever possible and write \linebreak $Y^{(\Delta)}_k, X_k^{(\Delta)}, \f(\cdot),\ldots$
instead of $Y^{(\Delta)}_k(\vartheta_0), X^{(\Delta)}_k(\vartheta_0), \f(\cdot, \vartheta_0),\ldots$.

To define the adjusted Whittle estimator and for the proof of the consistency of the Whittle
estimator we introduce the linear innovations of $Y^{(\Delta)}(\vartheta)$.

\begin{definition}	
	The linear innovations $\varepsilon^{(\Delta)}(\vartheta):=(\varepsilon_k^{(\Delta)}(\vartheta))_{k\in \N}$ of $Y^{(\Delta)}(\vartheta)$ are defined by \begin{align*}
	\varepsilon_k^{(\Delta)}(\vartheta)&=Y_k^{(\Delta)}(\vartheta)-\operatorname{Pr}_{k-1}(\vartheta)Y_k^{(\Delta)}(\vartheta), \quad \text{ where}\\
	\operatorname{Pr}_k(\vartheta)&=\text{ orthogonal projection onto }\mathcal M_k(\vartheta):=\overline {\operatorname{span}}\{Y^{(\Delta)}_\nu(\vartheta): -\infty<\nu\leq k\},
	\end{align*}
	where the closure is taken in the Hilbert space of random vectors with square-integrable components and inner product $(X,Y)\to \E \left[X^\top Y\right].$
\end{definition}

Adjusted to our notation, Proposition 2.1 of \cite{QMLE} gives the following representation of the linear innovations
of $Y^{(\Delta)}(\vartheta)$.

\begin{proposition}\label{invertible}	
	Suppose that the eigenvalues of $A(\vartheta)$ have strictly negative real parts and $\Sigma_L(\vartheta)$ is positive definite. Then, the following holds:
	\begin{itemize}
		\item[(a)] The Riccati equation
		\begin{eqnarray*}
			\Omega^{(\Delta)}(\vartheta)&=&e^{A(\vartheta)\Delta}\Omega^{(\Delta)}(\vartheta) \left(e^{A(\vartheta)\Delta}\right)^\top + \Sigma_N^{(\Delta)}(\vartheta)\\&&- \left( e^{A(\vartheta)\Delta}\Omega^{(\Delta)} (\vartheta)C(\vartheta)^\top\right)\left( C(\vartheta)\Omega^{(\Delta)}(\vartheta) C(\vartheta)^\top\right)^{-1}\left( e^{A(\vartheta)\Delta}\Omega^{(\Delta)}(\vartheta) C(\vartheta)^\top\right)^\top
		\end{eqnarray*} has a unique positive semidefinite solution $\Omega^{(\Delta)}(\vartheta)$.
		\item[(b)] Let
		$$K^{(\Delta)}(\vartheta)=\left( e^{A(\vartheta)\Delta}\Omega^{(\Delta)}(\vartheta) C(\vartheta)^\top\right)\left( C(\vartheta)\Omega^{(\Delta)}(\vartheta) C(\vartheta)^\top\right)^{-1}$$ be the Kalman gain matrix. Furthermore, define the polynomial $\Pi$  as
		\begin{eqnarray*} \label{Pi:rep}	
			\Pi(z,\vartheta):=\Pi^{(\Delta)}(z,\vartheta):=\left(I_m -C(\vartheta)\left(I_{N}-(e^{A(\vartheta)\Delta}-K^{(\Delta)}(\vartheta)C(\vartheta))z\right)^{-1}K^{(\Delta)}(\vartheta)z\right).
		\end{eqnarray*}
		Then, the linear innovations are
		\begin{eqnarray*}
			\varepsilon_{k}^{(\Delta)}(\vartheta)=\Pi(\mathsf{ B},\vartheta)Y^{(\Delta)}_k(\vartheta), \quad k\in\N.			
		\end{eqnarray*}
		Furthermore, the absolute value of any eigenvalue of $e^{A(\vartheta)\Delta}-K^{(\Delta)}(\vartheta)C(\vartheta)$ is less than one and $Y^{(\Delta)}(\vartheta)$ has the moving average representation
		\begin{eqnarray} \label{Pi-1}
		Y_k^{(\Delta)}(\vartheta)%
		=\varepsilon^{(\Delta)}_k(\vartheta)+C(\vartheta)\sum_{j=1}^{\infty}\left(e^{A(\vartheta)\Delta}\right)^{j-1}K^{(\Delta)}(\vartheta)\varepsilon^{(\Delta)}_{k-j}(\vartheta)
		=:\Pi^{-1}(\mathsf{ B},\vartheta)\varepsilon^{(\Delta)}_{k}(\vartheta). \quad\quad\quad
		\end{eqnarray}	
		\item[(c)] The covariance matrix $V^{(\Delta)}(\vartheta)$ of the linear innovations $\varepsilon^{(\Delta)}(\vartheta)$ has the representation $V^{(\Delta)}(\vartheta)=C(\vartheta)\Omega^{(\Delta)}(\vartheta) C(\vartheta)^\top.$ If $\Omega^{(\Delta)}(\vartheta)$ is positive definite and $C(\vartheta)$ has full rank, $V^{(\Delta)}(\vartheta)$ is invertible.
	\end{itemize}
\end{proposition}
Note that $\tr(V^{(\Delta)}(\vartheta))=\min_{X \in \mathcal{M}_{k-1}(\vartheta)}\E[(Y^{(\Delta)}_k(\vartheta)-X)^\top(Y^{(\Delta)}_k(\vartheta)-X)].$
An application of  \cite{BrockwellDavis}, Theorem 11.8.3, and \eqref{Pi-1} yield the representation
\begin{align}
\label{specY}
\f(\omega,\vartheta)=\Pi^{-1}(e^{-i\omega},\vartheta)\frac{V^{(\Delta)}(\vartheta)}{2\pi}\Pi^{-1}(e^{i\omega},\vartheta)^\top,
\quad  \omega \in [-\pi,\pi],
\end{align}
for the spectral density of $Y^{(\Delta)}(\vartheta)$.

\section{The Whittle estimator} \label{sec:WhittleEstimator}

\subsection{Consistency of the Whittle estimator} \label{sec:Cons:Whittle}

\begin{assumptionletterA}
	For all $\vartheta \in \Theta$ the following holds:\\[2mm]
	$\mbox{}$\,\,(A1) \,\, The parameter space $\Theta$ is a compact subset of $\R^r$.\\[1mm]
	$\mbox{}$\,\,(A2) \,\, $L(\vartheta)=(L_t(\vartheta))_{t\in \R}$ is a centered Lévy process with positive definite covariance matrix $\Sigma_L(\vartheta)$.\\[1mm]
	$\mbox{}$\,\,(A3) \,\, The eigenvalues of 	$A(\vartheta)$ have strictly negative real parts.\\[1mm]
	$\mbox{}$\,\,(A4) \,\, The functions $\vartheta \mapsto \Sigma_L(\vartheta),\ \vartheta \mapsto A(\vartheta),\ \vartheta\mapsto B(\vartheta)$ and $\vartheta\mapsto C(\vartheta)$ are continuous. In addition, \\
    $\mbox{}$\hspace*{1cm} $C(\vartheta)$ has full rank.\\[1mm]
	$\mbox{}$\,\,(A5) \,\, The linear state space model $(A(\vartheta),B(\vartheta),C(\vartheta),L(\vartheta))$ is minimal with McMillan degree $N$, i.e.,\\
 $\mbox{}$\hspace*{1cm} there exist no integer $\widetilde N<N$ and matrices $\tilde A \in \R^{\widetilde N\times \widetilde N},\ \widetilde B \in \R^{\widetilde N\times d}$ and $\widetilde C \in \R^{m \times \widetilde N}$ with\\
   $\mbox{}$\hspace*{1cm} $C(\vartheta)(z I_N -A(\vartheta))^{-1}B(\vartheta)=\widetilde C(z I_{\widetilde N} -\widetilde A)^{-1}\widetilde B$ for all $z \in \R$.\\[1mm] 
	$\mbox{}$\,\,(A6) \,\,
		For any  $\vartheta_1,\vartheta_2\in \Theta $ with $\vartheta_1\neq \vartheta_2$ there exists an $\omega \in [-\pi,\pi]$ such that
		$f_Y(\omega,\vartheta_1)\neq f_Y(\omega,\vartheta_2),$
     $\mbox{}$\hspace*{1cm} where $f_Y(\omega,\vartheta)$ is the spectral density of $Y(\vartheta)$.\\[1mm]
	$\mbox{}$\,\,(A7) \,\, The spectrum of $A(\vartheta)\in \R^{N\times N}$ is a subset of $\left\{z\in \C: -\frac{\pi}{\Delta}<\Im(z)<\frac\pi \Delta\right\}$.
\end{assumptionletterA}

\begin{remark}
	\label{identifizierbar}~
	\begin{itemize}
		\item[(a)] Note that Assumptions $(A2)$ and $(A3)$ allow us to calculate the linear innovations.
	Furthermore, the covariance matrix  $V^{(\Delta)}(\vartheta)$ of the linear innovations is non-singular (cf.  Lemma 3.14 in \cite{QMLE}). 
		\item[(b)] Theorem 2.3.4 in \cite{HannanDeistler} shows that $(A5)$ guarantees the uniqueness of the state space representation $(A(\vartheta),B(\vartheta), C(\vartheta), L(\vartheta))$ up to a change of basis. Hence, $(A5)$ reduces redundancies in the continuous-time model.
		In addition, \cite{QMLE}, Theorem 3.13, proved that Assumptions $(A2)$--$(A7)$ provide $\Delta$-identifiability of the collection of
		output processes $(Y(\vartheta), \vartheta \in \Theta)$, i.e., for fixed $\Delta>0$ and arbitrary $\vartheta_1,\vartheta_2 \in \Theta$ with $\vartheta_1\neq \vartheta_2$, there exists an $\omega \in [-\pi,\pi]$ with $f_Y^{(\Delta)}(\omega,\vartheta_1)\neq f_Y^{(\Delta)}(\omega,\vartheta_2).$
		\item[(c)] Assumptions $(A2)$ and $(A5)$ imply that $\Sigma_N^{(\Delta)}(\vartheta)$  has full rank.
		\item[(d)] Under Assumption A and representation~\eqref{spec} of the spectral density, the inverse  $f^{(\Delta)}_Y(\omega,\vartheta)^{-1}$ of the spectral density exists and the mapping $(\vartheta,\omega) \mapsto f^{(\Delta)}_Y({\omega},\vartheta)^{-1}$ is continuous.
	\end{itemize}
	
\end{remark}
We start to prove some auxiliary results which we need for the proof of the consistency of Whittle`s estimator. The following proposition states that the Whittle function $W_n$ converges almost surely uniformly.

{	\begin{proposition}\label{Hilfssatz1}
		Let Assumptions $(A1)$--$(A4)$ hold and  $$W(\vartheta):=\frac{1}{2\pi}\int_{-\pi}^{\pi}\operatorname{tr}\left(\f(\omega,\vartheta)^{-1}{\f(\omega)}\right)+\log\left(\det\left(\f(\omega,\vartheta)\right)\right)d\omega, \quad\vartheta \in\Theta. $$
		Then,
		\begin{align*}
		\sup_{\vartheta\in \Theta}\left|W_n(\vartheta)-W(\vartheta)\right|\overset{n\to \infty}{\longrightarrow}0 \quad \P\text{-a.s.}
		\end{align*}
	\end{proposition}
	Obviously, it is necessary that $\vartheta_0$ is a global minimum of  $W$   to guarantee the consistency of the Whittle estimator.
		\begin{proposition}\label{schritt2.2}
		Let Assumptions $(A1)$--$(A4)$ and $(A6)$ hold. Then,  $W$ has a unique global minimum in $\vartheta_0$.
	\end{proposition}
The proof is based on an alternative representation of $W$. Namely, the function $W$ is exactly  the limit function of the quasi maximum likelihood estimator of \cite{QMLE}.

	\begin{lemma}\label{schritt2}
		Let Assumptions $(A1)$--$(A4)$ hold and let $\xi_k^{(\Delta)}(\vartheta)=\Pi(\mathsf{B},\vartheta)Y_k^{(\Delta)}$ with $\Pi(z,\vartheta)$ as given in  \Cref{invertible}.
		Furthermore, define
		$$  \mathcal{L}(\vartheta):=\E\left[\operatorname{tr}\left(\xi_1^{(\Delta)}(\vartheta)^\top V^{(\Delta)}(\vartheta)^{-1}\xi_{1}^{(\Delta)}(\vartheta)\right)\right]+\log(\det(V^{(\Delta)}(\vartheta)))-m\log(2\pi), \quad\vartheta \in\Theta.$$
		Then,	$W(\vartheta)= \mathcal{L}(\vartheta)$ for $\vartheta \in\Theta$.
	\end{lemma}

	Finally, we are able to state the first main result of this paper, which gives the consistency of the Whittle estimator.

	\begin{theorem1}\label{satz1}
		Let Assumption A hold. Then, as $n\to \infty$, $$\widehat{\vartheta}_n^{(\Delta)}\overset{a.s.}{\longrightarrow}\vartheta_0.$$
	\end{theorem1}

	\subsection{Asymptotic normality of the Whittle estimator} \label{sec:Whittle:asymptoticNormal}

	For the asymptotic normality of the Whittle estimator some further assumptions are required.

	\begin{assumptionletterB}~\label{AssB}\\[2mm]
			$\mbox{}$\,\,(B1) \,\, The true parameter value $\vartheta_0$ is in the interior of $\Theta$.\\[1mm]
			$\mbox{}$\,\,(B2) \,\, $\E\|L_1\|^4<\infty.$\\[1mm]
			$\mbox{}$\,\,(B3) \,\, The functions $\vartheta \mapsto A(\vartheta),\ \vartheta\mapsto B(\vartheta)$, $\vartheta\mapsto C(\vartheta)$ and $\vartheta \mapsto \Sigma_L(\vartheta)$ are three times continuously\\
    $\mbox{}$\hspace*{1cm}differentiable.\\[1mm]
			$\mbox{}$\,\,(B4) \,\, For any $c \in \C^r$, there exists an $\omega^* \in [-\pi,\pi]$ such that $\nabla_\vartheta \f(\omega^*,\vartheta_0)c \neq 0_{m^2}.$
	\end{assumptionletterB}
	
	\begin{remark}\label{BemAsyNorm}~
		Due to representation \eqref{spec} of the spectral density, under Assumption A and $(B3)$ the mapping $\vartheta\mapsto \f(\omega,\vartheta)$ is three times continuously differentiable.
	\end{remark}
	
	The proof of the asymptotic normality of the Whittle estimator is based on a Taylor expansion
	of $\nabla_\vartheta W_n$  around $\widehat \vartheta_n^{(\Delta)}$ in $\vartheta_0$, i.e.,
	\begin{align}\label{taylor}
	\sqrt{n} \left[\nabla_\vartheta W_n( \vartheta_0)\right]=\sqrt n \left[\nabla_\vartheta W_n(\widehat \vartheta_n^{(\Delta)})\right]-\sqrt n (\widehat \vartheta^{(\Delta)}_n-\vartheta_0)^\top\left[\nabla_\vartheta^2 W_n(\vartheta^*_n)\right]
	\end{align}
	for an appropriate $\vartheta_n^*\in\Theta$ with $\|\vartheta_n^*-\vartheta_0\|\leq \|\widehat\vartheta^{(\Delta)}_n-\vartheta_0\|$.
	Since $\widehat \vartheta_n^{(\Delta)}$ minimizes $W_n$ and converges almost surely to $\vartheta_0$, which is in the interior of $\Theta$ (Assumption $(B1)$), $\nabla_\vartheta W_n(\widehat \vartheta_n^{(\Delta)})=0$.
	Hence, in the case of an invertible matrix $\nabla_\vartheta^2 W_n(\vartheta^*_n)$ we can rewrite  \eqref{taylor} and obtain
	\begin{align}\label{taylor1}
	\sqrt n (\widehat \vartheta^{(\Delta)}_n-\vartheta_0)^\top=-\sqrt{n} \left[\nabla_\vartheta W_n( \vartheta_0)\right]\left[\nabla_\vartheta^2 W_n(\vartheta^*_n)\right]^{-1}.
	\end{align}
	Therefore, we receive the asymptotic normality of the Whittle estimator from the asymptotic behavior of the individual components in \eqref{taylor1}.
	First, we investigate the asymptotic behavior of the Hessian matrix $\nabla_\vartheta^2 W_n(\vartheta^*_n)$.
	
	\begin{proposition}\label{remarkunif}
		Let Assumptions $(A1)$--$(A4)$ and $(B3)$ hold and
		\beam  \label{Sigma_2}
		\Sigma_{\nabla^2 W}=\frac{1}{2\pi}\int_{-\pi}^{\pi} \nabla_\vartheta \f(-\omega,\vartheta_0)^{\top}\left[\f(-\omega)^{-1}\otimes \f(\omega)^{-1}\right]\nabla_\vartheta\f(\omega,\vartheta_0)  d\omega.
		\eeam
		Furthermore, let $(\vartheta_n^*)_{n\in \N}$ be a sequence in $\Theta$ with $\vartheta_n^*\overset{a.s.}{\longrightarrow}\vartheta_0$ as $n\to\infty$. Then, as $n\to\infty$,
		$$ \nabla^2_\vartheta W_n(\vartheta^*_n)\overset{a.s.}{\longrightarrow}\Sigma_{\nabla^2 W}.$$
	\end{proposition}
	
	Further, we require that for large $n$ the random matrix $\nabla_\vartheta^2 W_n(\vartheta^*_n)$ is invertible.
	Therefore, we show the positive definiteness of the limit matrix $\Sigma_{\nabla^2 W}$.
	
	\begin{lemma} \label{positiveSigmanabla2}
		Let Assumptions A and $(B4)$ hold.
		Then,  $\Sigma_{\nabla^2 W}$ is positive definite.
	\end{lemma}
	
	\begin{remark}
		For Gaussian state space processes
		\begin{align*}
			J=&\left[2\E\left[\left(\frac{\partial}{\partial \vartheta_i}\varepsilon_1^{(\Delta)}(\vartheta_0)\right)^\mathsf{T} V^{(\Delta)-1} \left(\frac{\partial}{\partial \vartheta_j}\varepsilon^{(\Delta)}_1(\vartheta_0)\right)\right]\right.
			\\ &\quad\quad
			\left.+\tr\left(\left(\frac{\partial}{\partial \vartheta_i}V^{(\Delta)}({\vartheta_0})\right)V^{(\Delta)-1}
		\left( \frac{\partial}{\partial \vartheta_j}V^{(\Delta)}({\vartheta_0})\right)V^{(\Delta)-1}\right)\right]_{i,j=1,\ldots,r}
		\end{align*}
		is the Fisher information matrix (cf. \cite{QMLE}). Since $W(\vartheta)=\mathcal{L}(\vartheta)$ due to \Cref{schritt2}, and
		$\nabla_\vartheta\f(\omega,\vartheta)$ is uniformly bounded by an integrated dominant,  we get
		by some straightforward applications of dominated convergence and some arguments of the proof of  \cite{QMLE}, Lemma~2.17,
		that
		\beao
		J[{i,j}]&=\lim_{n\to\infty}\E\left[\frac{\partial}{\partial \vartheta_i}\frac{\partial}{\partial \vartheta_j}\mathcal L_n(\vartheta_0)\right]=\frac{\partial}{\partial \vartheta_i}\frac{\partial}{\partial \vartheta_j}\lim_{n\to\infty}\E[
		\mathcal L_n(\vartheta_0)]
		\\
		&=\frac{\partial}{\partial \vartheta_i}\frac{\partial}{\partial \vartheta_j}W(\vartheta_0)=\lim_{n\to\infty}\E\left[\frac{\partial}{\partial \vartheta_i}\frac{\partial}{\partial \vartheta_j} W_n(\vartheta_0)\right]=\Sigma_{\nabla^2 W}[i,j],
		\eeao
		where $\mathcal{L}_n(\vartheta)$ is the quasi-Gaussian likelihood function.
		Furthermore, \cite{QMLE},\\ Lemma~2.17, show that if  Assumption A holds
		and  if there exists an $j_0\in\N$ such that the $((j_0+2)m^2)\times r$-matrix
		\begin{align*}
		\nabla\left[\begin{array}{c c}
		&\left[I_{j_0+1}\otimes K^{(\Delta)}(\vartheta_0)^\top\otimes C(\vartheta_0)\right] \left[\left(\vecc\left( e^{I_N\Delta}\right)\right)^\top\left(\vecc\left( e^{A(\vartheta_0)\Delta}\right)\right)^\top\cdots \left(\vecc\left( e^{A^j(\vartheta_0)\Delta}\right)\right)^\top\right]^\top\\
		& \vecc\left(V^{(\Delta)}(\vartheta_0)\right)\end{array}\right]
		\end{align*}
		has rank $r$,
		then the matrix $J$ is positive definite. Thus, our assumption (B4) can be replaced
		by this condition.		
	\end{remark}
	
	Next, we investigate the asymptotic behavior of the second term in \eqref{taylor1}.
	Since the components of the score $\nabla_\vartheta W_n(\vartheta_0)$ can be written as an integrated periodogram,
	we first derive the asymptotic behavior of the integrated periodogram and state the
	asymptotic normality afterwards.

	\begin{proposition}\label{Prop2}~
		Let Assumptions $(A2)$--$(A4)$ and $(B2)$ hold.
		Suppose $\eta:[-\pi,\pi]\to \C^{m\times m}$ is a symmetric matrix-valued continuous function with Fourier coefficients $(\mathfrak{f}_u)_{u\in \Z}$ satisfying \linebreak
		$\sum_{u=-\infty}^{\infty} \|\mathfrak{f}_u\||u|^{1/2}<\infty$.
		Then, as $n\to \infty$,
		$$\frac{1}{2\sqrt n}\sum_{j=-n+1}^{n}\tr\left(\eta(\omega_j)I_n(\omega_j)- \eta(\omega_j)\f(\omega_j) \right) \overset{\mathcal D}{\longrightarrow}\mathcal N(0,\Sigma_\eta),$$ where \begin{eqnarray*}
			\Sigma_\eta
			&=&\frac{1}{\pi}\int_{-\pi}^{\pi}\tr\left(\eta(\omega)\f(\omega)\eta(\omega) \f(\omega)\right)d\omega+ \frac{1}{16\pi^4}\int_{-\pi}^{\pi}\vecc\left(\Phi(e^{-i\omega})^\top\eta(\omega)^\top\Phi(e^{i\omega})\right)^\top d\omega\\
			&&\left(\E\left[N_1^{(\Delta)}N_1^{(\Delta)\top}\otimes N_1^{(\Delta)}N_1^{(\Delta)\top}\right]-3\Sigma_N^{(\Delta)}\otimes \Sigma_N^{(\Delta)}\right)\int_{-\pi}^{\pi}\vecc\left(\Phi(e^{i\omega})^\top\eta(\omega)\Phi(e^{-i\omega})\right)d\omega.
		\end{eqnarray*}
	\end{proposition}
	
	The asymptotic behavior of the integrated periodogram is interesting for its own. It can be modified to derive goodness
	of fit tests for state space models which are continuous functionals of the integrated periodogram (cf. \cite{Priestley:1981}).
	
	\begin{remark} \label{Remark:3.11}
		Let the driving L\'evy process be a Brownian motion. Since the fourth moment of a centered normal distribution is equal to three times its second moment and $N_1^{(\Delta)}\overset{\mathcal{D}}{\sim} \mathcal{N}(0,\Sigma_N^{(\Delta)}),$ we get $\E[	N_1^{(\Delta)}N_1^{(\Delta)\top}\otimes N_1^{(\Delta)}N_1^{(\Delta)\top}]=3\Sigma_N^{(\Delta)}\otimes \Sigma_N^{(\Delta)}.$  Therefore, the matrix $\Sigma_\eta$ in \Cref{Prop2}
		reduces to
		\begin{eqnarray*}
			\Sigma_\eta
			&=&\frac{1}{\pi}\int_{-\pi}^{\pi}\tr\left(\eta(\omega)\f(\omega)\eta(\omega) \f(\omega)\right)d\omega,
		\end{eqnarray*}
		which is for $m=1$ equal to $\Sigma_\eta =\frac{1}{\pi}\int_{-\pi}^{\pi} \eta(\omega)^2\f(\omega)^2 d\omega.$
	\end{remark}
	
	Finally, we obtain the asymptotic behavior of the score function.
	
	\begin{proposition}\label{abl1W}
		Let Assumptions $(A2)$--$(A4)$ and $(B2)$--$(B3)$ hold. Define
		\begin{eqnarray}  \label{SigmaWnabla}
		\lefteqn{\Sigma_{\nabla W}=\frac{1}{\pi}\int_{-\pi}^{\pi} \nabla_\vartheta \f(-\omega,\vartheta_0)^{\top}\left[\f(-\omega)^{-1}\otimes \f(\omega)^{-1}\right]\nabla_\vartheta\f(\omega,\vartheta_0)  d\omega} \nonumber\\
		&&\qquad +  {\frac{1}{16\pi^4}\left[\int_{-\pi}^{\pi}\left[\Phi(e^{i\omega})^\top \f(\omega)^{-1}\otimes \Phi(e^{-i\omega})^\top \f(-\omega)^{-1} \right]\nabla_\vartheta\f(-\omega,\vartheta_0) d\omega\right]^\top} \nonumber\\
		&&\qquad\qquad\qquad \cdot  {\left[\E\left[N_1^{(\Delta)}N_1^{(\Delta)\top}\otimes N_1^{(\Delta)}N_1^{(\Delta)\top}\right]-3\Sigma_N^{(\Delta)}\otimes \Sigma_N^{(\Delta)}\right]}\nonumber \\
		&&\qquad\qquad\qquad \cdot {\left[\int_{-\pi}^{\pi}\left[\Phi(e^{-i\omega})^{\top}\f(-\omega)^{-1}\otimes \Phi(e^{i\omega})^\top \f(\omega)^{-1} \right]\nabla_\vartheta\f(\omega,\vartheta_0) d\omega\right]}.
		\end{eqnarray}
		Then, as $n\to \infty$, $$\sqrt n \left[\nabla_\vartheta W_n(\vartheta_0)\right]\overset{\mathcal D}{\longrightarrow} \mathcal N(0,\Sigma_{\nabla W}).$$
	\end{proposition}
	
	Now, we are able to present the main result of this paper, the asymptotic normality of the Whittle estimator.
	
	\begin{theorem1}\label{SatzAsyN}
		Let Assumptions $A$ and $B$ hold. Furthermore, let $\Sigma_{\nabla W}$ be defined as in  \eqref{SigmaWnabla}
		and $\Sigma_{\nabla^2 W}$ be defined as in \eqref{Sigma_2}.
		Then, as $n\to\infty$,
		$$\sqrt{n}\left(\widehat{\vartheta}_n^{(\Delta)}-\vartheta_0\right)\overset{\mathcal D}{\longrightarrow}\mathcal{N} (0,\Sigma_W),$$ where $\Sigma_W$ has the representation $\Sigma_W=[\Sigma_{\nabla^2 W}]^{-1}\Sigma_{\nabla W}[\Sigma_{\nabla^2 W}]^{-1}.$
	\end{theorem1}
	In contrast to the quasi maximum likelihood
	estimator of \cite{QMLE},  the limit covariance matrix of the Whittle estimator has an analytic representation.
	It can be used for the calculation of confidence bands.
	
	\begin{remark} \label{Remark 3.6}
		We want to compare our outcome with an analogue result for stationary discrete-time \linebreak VARMA$(p,q)$ processes $(Z_n)_{n\in\N}$  of the form \eqref{VARMA} with finite fourth moments.
		In our setting we have the drawback that the autoregressive
		and the moving average polynomial influence the covariance matrix $\Sigma_N^{(\Delta)}$ of $(N^{(\Delta)}_k)_{k\in\N_0}$.
		In the setting of stationary VARMA$(p,q)$ processes of \cite{DunsmuirHannan76}
		the covariance matrix $\Sigma_e$ of the white noise $(e_n)_{n\in\Z}$ is not affected by the AR and MA polynomials.
		It was shown in  \cite{DunsmuirHannan76} that under very general assumptions for $d=m$ the resulting limit covariance matrix of the Whittle estimator for the VARMA parameters has the representation
		\begin{eqnarray*}
			\Sigma_W^{\text{VARMA}}	=\left[\frac{1}{4\pi}\int_{-\pi}^{\pi}\nabla_\vartheta f_Z(-\omega,\vartheta_0)^{\top}\left[f_Z(-\omega)^{-1}\otimes f_Z(\omega)^{-1}\right]\nabla_\vartheta f_Z(\omega,\vartheta_0)d\omega\right]^{-1}
			=2 \cdot [\Sigma_{\nabla^2W}^{\text{VARMA}}]^{-1},
		\end{eqnarray*}
		which is simpler  than our $\Sigma_W$.
		This can be traced back to
		$\Sigma_{\nabla W}^{\text{VARMA}}=2\cdot \Sigma_{\nabla^2W}^{\text{VARMA}},$
		which is motivated on p.~\pageref{Subsection:Remark:3.6}.
		In particular, for a Gaussian VARMA model, $ \Sigma_W^{\text{VARMA}}$
		is the inverse of the Fisher information matrix.
	\end{remark}

	\begin{remark}~\label{Remark 3.15}
		\begin{enumerate}
			\item[(a)] Let the driving L\'evy process be a Brownian motion.  Due to \Cref{Remark:3.11}, the matrix $\Sigma_{\nabla W}$ reduces to
			\begin{align*}
			\Sigma_{\nabla W}=\frac{1}{\pi}\int_{-\pi}^{\pi}\nabla_\vartheta \f(-\omega,\vartheta_0)^{\top}\left[\f(-\omega)^{-1}\otimes \f(\omega)^{-1}\right]\nabla_\vartheta\f(\omega,\vartheta_0)d\omega= 2\cdot[\Sigma_{\nabla^2 W}]^{-1},
			\end{align*}
			and hence, $\Sigma_W=2\cdot [\Sigma_{\nabla W}]^{-1}$ is the inverse of the Fisher information matrix and corresponds
			to $\Sigma_W^{\text{VARMA}}$ as in the previously mentioned discrete-time VARMA setting.
			\item[(b)] Let $d=m=N$ and $C(\vartheta)=I_m$. Then, the state space model is a multivariate Ornstein-Uhlenbeck process (MCAR(1) process).
			In this example, $\Sigma_{\nabla W}=2\cdot[\Sigma_{\nabla^2 W}]^{-1}$ holds as well.  Because of $\Phi(z,\vartheta)=\sum_{j=0}^\infty e^{A(\vartheta)\Delta j}z^j=(1-e^{A(\vartheta)\Delta}z)^{-1}=\Pi^{-1}(z,\vartheta)$, the arguments are very similar to the arguments for VARMA models in  \Cref{Remark 3.6}.
		\end{enumerate}
	\end{remark}
	
	\section{The adjusted Whittle estimator} \label{sec:adjustedWhittleestimator}

	In the following, we solely consider state space models where $Y$ and $L$ are one-dimensional, i.e., $A\in \R^{N\times N}$, $B \in \R^{N\times 1}$ and $C\in \R^{1\times N}$. This includes, in particular, univariate CARMA processes, see, e.g., \cite{Brockwell:Lindner:2009,Brockwell:2014} for the explicit definition and existence criteria.
	Further, we assume that the variance parameter $\sigma^2_L$ of the driving L\'evy process does not depend on $\vartheta$ and has not to be estimated.
	In this context, we consider an adjusted Whittle estimator which takes into account that we do not have to estimate
	the variance. Such adjusted Whittle estimators are useful for the estimation
	of heavy tailed CARMA models with infinite variance. For example, \cite{Mikoschetal:1995} estimate the parameters
	of ARMA models in discrete time whose noise has a symmetric stable distribution. In some future work we will
	investigate such an adjusted Whittle estimator for heavy tailed models as well.
	
Now, the Whittle function is adapted in a way which makes it independent of the variance of the driving Lévy process.
Therefore, we use the representation of the spectral density in \eqref{specY}. 
Although the
variance $\sigma_L^2$ goes linearly in $\Omega^{(\Delta)}(\vartheta)$ and $V^{(\Delta)}(\vartheta)$, both $K^{(\Delta)}(\vartheta)$
and $\Pi(z,\vartheta)$ do not depend on $\sigma_L^2$ anymore.
The second summand of the Whittle function $W_n$ is removed and the first term is adjusted so that we obtain the \textit{adjusted Whittle function}
	\begin{align*}
	W^{(A)}_n(\vartheta)
	&=\frac{\pi}{n}\sum_{j=-n+1}^{n}|\Pi(e^{i\omega_j},\vartheta)|^{2}I_n(\omega_j)=\frac {V^{(\Delta)}(\vartheta)} {2n}\sum_{j=-n+1}^{n}\f(\omega_j,\vartheta)^{-1}I_n(\omega_j).
	\end{align*}
	The corresponding minimizer $$\widehat\vartheta^{(\Delta,A)}_n=\arg\min_{\vartheta \in \Theta}W^{(A)}_n(\vartheta)$$ is the \textit{adjusted Whittle estimator}.

	\subsection{Consistency of the adjusted Whittle estimator} \label{sec:adWhittle:consistency}
	
	Since the estimation procedure is different to that of the previous sections, we have to adjust Assumption~A.

\begin{assumptionletterC}
	Let Assumptions $(A1)$--$(A5)$ and $(A7)$ hold. Furthermore, assume\\[2mm]
	$\mbox{}$\,\,($\widetilde A$6) \,\, For any $\vartheta_1,\ \vartheta_2 \in \Theta$, $\vartheta_1\neq \vartheta_2$, there exists some $z\in\C$ with $|z|=1$ and $\Pi(z,\vartheta_1)\neq \Pi(z,\vartheta_2)$.
\end{assumptionletterC}
	
	It is needless to say that conditions as those for the function $\vartheta \to \sigma_L^2$ are fulfilled naturally.
	In addition to \Cref{identifizierbar}, which remains mostly applicable, we stress that,  under Assumption $\widetilde A$, $\Pi^{-1}$ as defined in \eqref{Pi-1} exists for all $\vartheta\in \Theta$ and that the mapping $(\omega,\vartheta)\to \Pi^{-1}(e^{i\omega},\vartheta)$ is continuous.

	\begin{theorem1}\label{consistency2}
		Let Assumption $\widetilde A$ hold. Then, as $n\to\infty$,
		$$\widehat\vartheta^{(\Delta,A)}_n\overset{a.s.}{\longrightarrow}\vartheta_0.$$
	\end{theorem1}
	The proof follows the same steps as the proof for the
	consistency of the Whittle estimator in \Cref{satz1}.
	
	\subsection{Asymptotic normality of the adjusted Whittle estimator} \label{sec:adWhittle:normality}
	
	For the asymptotic normality of the adjusted Whittle estimator we have to adapt Assumption B.
	
	\bigskip
	
\begin{assumptionletterD}
	Let Assumptions (B1)-(B3) hold. Furthermore, assume\\[2mm]
	$\mbox{}$\,\, ($\widetilde B$4) \,\, For any $c \in \C^r$ there exists an $\omega^* \in [-\pi,\pi]$ such that $\nabla_\vartheta |\Pi(e^{i\omega^*},\vartheta_0)|^{-2}c \neq 0.$
	\end{assumptionletterD}

	\begin{remark}\label{remarkadjWE}
		Under Assumption $\widetilde A$ and Assumption $\widetilde B$ the mapping $\vartheta\to \Pi(e^{i\omega},\vartheta)$ is three times continuously differentiable. Similarly to~\Cref{positiveSigmanabla2}, $(\widetilde B4)$ guarantees the invertibility of
		\beam \label{Sigma_2_a}
		\Sigma_{\nabla^2 W^{(A)}} :=\frac{V^{(\Delta)}}{2\pi}\int_{-\pi}^{\pi}\nabla_\vartheta \log\left(|\Pi(e^{i\omega},\vartheta_0)|^{-2}\right)^\top\nabla_\vartheta \log\left(|\Pi(e^{i\omega},\vartheta_0)|^{-2}\right) d\omega.
		\eeam
	\end{remark}

	\begin{theorem1}\label{normality2}
		Let Assumption $\widetilde A$ and $\widetilde B$ hold.
		Further, let $\Sigma_{\nabla^2 W^{(A)}}$ be defined as in \eqref{Sigma_2_a} and
		\begin{eqnarray*}
			\Sigma_{\nabla W^{(A)}}&=&\frac{V^{(\Delta)2}}{\pi}\int_{-\pi}^{\pi}\nabla_\vartheta \log\left(|\Pi(e^{i\omega},\vartheta_0)|^{-2}\right)^\top\nabla_\vartheta \log\left(|\Pi(e^{i\omega},\vartheta_0)|^{-2}\right) d\omega  \\
			&&+ \frac{1}{4\pi^2}\left[\int_{-\pi}^{\pi}\nabla_\vartheta|\Pi(e^{i\omega},\vartheta_0)|^{2\top}\left[\Phi(e^{i\omega})\otimes \Phi(e^{-i\omega}) \right] d\omega\right]\\ &&\quad\quad\cdot{\left[\E\left[N_1^{(\Delta)}N_1^{(\Delta)\top}\otimes N_1^{(\Delta)}N_1^{(\Delta)\top}\right]-3\Sigma_N^{(\Delta)}\otimes \Sigma_N^{(\Delta)}\right]}\\
			&&\quad\quad\cdot {\left[\int_{-\pi}^{\pi} \nabla_\vartheta|\Pi(e^{i\omega},\vartheta_0)|^{2\top}\left[\Phi(e^{-i\omega})\otimes \Phi(e^{i\omega}) \right] d\omega\right]^\top}.
		\end{eqnarray*}
		Then, as $n\to \infty$, $$\sqrt{n}\left(\widehat{\vartheta}_n^{(\Delta,A)}-\vartheta_0\right)\overset{\mathcal D}{\longrightarrow}\mathcal{N} (0,\Sigma_{W^{(A)}}),$$ where $\Sigma_{W^{(A)}}$ has the representation $\Sigma_{W^{(A)}}=[\Sigma_{\nabla^2 W^{(A)}}]^{-1}\Sigma_{\nabla W^{(A)}}[\Sigma_{\nabla^2 W^{(A)}}]^{-1}.$
	\end{theorem1}

	\begin{remark} For the one-dimensional CAR(1) (Ornstein-Uhlenbeck) process, for which $m=d=N=1$ and
			$C(\vartheta)=B(\vartheta)=1$ holds, the limit covariance matrix $\Sigma_{W^{(A)}}$ of ~\Cref{normality2} reduces due to \Cref{OU_univariat} in the Supplementary Material
			and Theorem 3'{}'{}', Chapter 3, of \cite{Hannanmult}  to
			\beao
			\Sigma_{W^{(A)}}&=&4\pi \left[\int_{-\pi}^{\pi}\nabla_\vartheta \log\left(|\Pi(e^{i\omega},\vartheta_0)|^{-2}\right)^\top\nabla_\vartheta \log\left(|\Pi(e^{i\omega},\vartheta_0)|^{-2}\right) d\omega\right]^{-1}\\
			&=&4\pi\left[\int_{-\pi}^{\pi}\nabla_\vartheta \log(\f(\omega, \vartheta_0))^\top\nabla_\vartheta\log(\f(\omega, \vartheta_0))d\omega\right.\\
			&&\left.-\frac 1 {2\pi}\left(\int_{-\pi}^{\pi}\nabla_\vartheta \log(\f(\omega, \vartheta_0))d\omega\right)^\top \left(\int_{-\pi}^{\pi}\nabla_\vartheta \log(\f(\omega, \vartheta_0))d\omega\right)\right]^{-1}.
		\end{eqnarray*}
		Due to \Cref{Remark 3.15} (b)
		\begin{eqnarray*}
			\Sigma_{W}&=&2\cdot[\Sigma_{\nabla^2 W}]^{-1}= 4\pi\left[\int_{-\pi}^{\pi}\nabla_\vartheta \log(\f(\omega, \vartheta_0))^\top\nabla_\vartheta\log(\f(\omega, \vartheta_0))d\omega\right]^{-1}
		\end{eqnarray*}
		and hence, $\Sigma_{W^{(A)}}\geq \Sigma_{W}$. Thus, the adjusted Whittle estimator has a higher variance than the Whittle estimator.
		 Let $\vartheta_0<0$ be the zero of the AR polynomial in the CAR(1) model, i.e., $A(\vartheta_0)=\vartheta_0$. Simple calculations show that $\Sigma_{W^{(A)}}=e^{-2\vartheta_0}-1$
		which is equal to the asymptotic variance of the maximum likelihood estimator  of \cite{Brockwell:Lindner:2019}.
	 However, it is not possible to make this conclusion for general CARMA processes. There exist CARMA processes
		for which the adjusted Whittle estimator  has a different asymptotic variance than the maximum likelihood estimator  of \cite{Brockwell:Lindner:2019}.
\end{remark}

\section{Simulation} \label{sec:simulation}

In this section, we show the practical applicability of the Whittle and the adjusted Whittle estimator.
We simulate continuous-time state space models with an Euler-Maruyama scheme for differential equations with initial value $X(0)=Y(0)=0$ and step size $0.01$. Using $\Delta=1$ and the interval $[0,500]$, we therefore get $n_1=500$ discrete observations. Furthermore, we investigate how the results change qualitatively when we consider the intervals $[0,2000]$ and $[0,5000]$, which imply $n_2=2000$ and $n_3=5000$ observations, respectively. In each sample, we use 500 replicates.
We investigate the estimation procedure based on two different driving L\'evy processes. Since the Brownian motion is the most common Lévy process, we examine Whittle`s estimation based on a Brownian motion. As a second case, we analyze the performance based on a bivariate normal-inverse Gaussian (NIG) L\'evy process, which is often used in modeling stochastic volatility or stock returns, see \cite{barndorff1997normal}.
The resulting increments of this process are characterized by the density $$f(x,\mu,\alpha,\beta,\delta_{NIG},\Delta_{NIG})=\frac{\delta_{NIG}}{2\pi}\frac{(1+\alpha g(x))}{g(x)^3}\exp({\delta_{NIG}\kappa}+ \beta^\top x-\alpha g(x)), \quad x \in \R^2,$$ with $$g(x)=\sqrt{\delta_{NIG}^2+\langle x-\mu, \Delta_{NIG}(x-\mu)\rangle},\ \kappa^2=\alpha^2-\langle \beta, \Delta_{NIG}\beta\rangle>0.$$ Thereby, $\beta \in \R^2$ is a symmetry parameter, $\delta_{NIG}\geq 0$ is a scale parameter and the positive definite matrix $\Delta_{NIG}$ models the dependency between the two components of the bivariate Lévy process $(L_t)_{t\in\R}$. We set $\mu=-({\delta_{NIG}\Delta_{NIG}\beta})/{\kappa}$ to guarantee that the resulting L\'evy process is centered, see, e.g., \cite{oigaard2005estimation} or \cite{barndorff1997normal} for more details.
For better comparability of the Brownian motion case and the NIG L\'evy process case, we choose the parameters
of the NIG Lévy process in a way that the resulting covariance matrices of the Lévy processes are the same.

The performances of the Whittle and the adjusted Whittle estimator are compared with the well known quasi maximum likelihood estimator (QMLE) presented in \cite{QMLE}.
The assumptions concerning the QMLE of \cite{QMLE} are the same as ours. Therefore, the  Echelon canonical form given in \cite{QMLE}, Section 4, is used as parametrization (cf. \cite{guidorzi1975canonical}) which is standard for state space and VARMA models (cf. \cite{HannanDeistler}).
In particular, Assumptions $(A1)$--$(A7)$ and $(B1)$--$(B3)$ are satisfied.

In the multivariate setting, we consider bivariate MCARMA(2,1) processes of the form
\begin{align*}
dX_t(\vartheta)= A(\vartheta)X_t(\vartheta)dt+ B(\vartheta)dL_t(\vartheta)\quad \text{ and }\quad Y_t(\vartheta)= C(\vartheta)X_t(\vartheta), \quad t\geq 0,
\end{align*}
with \begin{align*}A(\vartheta)&=\left( \begin{array}{c c c}\vartheta_1 & \vartheta_2 &0 \\ 0&0&1 \\ \vartheta_3 & \vartheta_4 &\vartheta_5\end{array}\right),\quad \quad B(\vartheta)=\left( \begin{array}{c c}\vartheta_1 & \vartheta_2  \\ \vartheta_6 & \vartheta_7\\ \vartheta_3+\vartheta_5 \vartheta_6 & \vartheta_6 +\vartheta_5\vartheta_7\end{array}\right),\\
C(\vartheta)&=\left( \begin{array}{c c c}1 & 0 &0 \\ 0 & 1 &0 \end{array}\right),\quad \quad  \quad \quad \Sigma_{L}(\vartheta)=\left( \begin{array}{ c c}\vartheta_8 & \vartheta_9 \\ \vartheta_9 & \vartheta_{10} \end{array}\right).\end{align*}
This parametrization is given in Table~1 of \cite{QMLE} and the representations of the corresponding AR polynomial $P$ and MA polynomial $Q$
are given in Table~2 of that paper. Furthermore,  we get the order $(2,1)$ of the MCARMA process from there as well.
In our example, the true parameter value  is
$$\vartheta^{(1)}_0=(-1,-2,1,-2,-3,1,2,0.4751,-0.1622,0.3708).$$
To generate a NIG L\'evy process with the same covariance matrix, we rely on the parameters
$$\delta_{NIG}^{(1)}=1, \quad \alpha^{(1)}=3,\quad \beta^{(1)}=(1,1)^T, \quad \Delta^{(1)}_{NIG}=\left( \begin{array}{ c c}5/4 & -1/2 \\ -1/2 & 1 \end{array}\right).$$
The estimation results  are summarized in ~\Cref{table_1} and \Cref{table_2} for the Brownian motion driven model and the NIG driven model, respectively. The consistency can be observed in all simulations, namely the bias and the standard deviations are decreasing for increasing sample size for both the Whittle estimator and the quasi maximum likelihood estimator.
The performance of the estimators is very similar.

\begin{table}[h]
	\begin{center}
		\begin{tabular}{|c||c|c|c||c|c|c|}\hline
			\multicolumn{7}{|c|}{$n_1=500$} \\\hline
			&  \multicolumn{3}{|c||}{Whittle} & \multicolumn{3}{|c|}{QMLE}\\ \hline
			\hspace*{0.1cm} $\vartheta_0$\hspace*{0.1cm} &mean & bias & std.  & mean  & bias &std.   \\ \hline
			-1 &-0.9969& 0.0031& 0.0325& -1.0012& 0.0012& 0.0572 \\
			-2 &-2.0218&0.0218&0.0582& -2.0128& 0.0128& 0.0689 \\
			1 &0.9980&0.0020&0.0520& 1.0075& 0.0075&0.0722 \\
			-2 &-2.0498& 0.0498&0.1060& -1.9797& 0.0203& 0.0758\\
			-3 &-2.9840&0.0160&0.0498&-2.9913& 0.0087& 0.0907  \\
			1 &1.0062&0.0062&0.1309&0.8034& 0.1966&0.3896  \\
			2 & 1.9983&0.0017&0.0532& 2.0036& 0.0036&0.0768 \\ \hline
			0.4751 &0.4746&0.0005&0.0407& 0.4693& 0.0048& 0.0691  \\
			-0.1622 &-0.1629&0.0007&0.0134&-0.1624&  0.0002&0.0405  \\
			0.3708 &0.3706&0.0002&0.0064& 0.3712 & 0.0004& 0.0328 \\ \hline \hline
			\multicolumn{7}{|c|}{$n_2=2000$} \\\hline
			&  \multicolumn{3}{|c||}{Whittle} & \multicolumn{3}{|c|}{QMLE}\\ \hline
			\hspace*{0.1cm} $\vartheta_0$\hspace*{0.1cm} &mean  & bias & std.  & mean  & bias &std.   \\ \hline
			-1 &-0.9970&0.0030& 0.0155			&-0.9957& 0.0043& 0.0260 \\
			-2 & -2.0062&0.0062&0.0252			& -2.0047& 0.0047& 0.0350 \\
			1 & 0.9909 &0.0091& 0.0266			& 1.0038& 0.0038& 0.0399 \\
			-2 & -2.0394 &0.0394&0.0501			& -2.0122&0.0122&0.0481\\
			-3 & -2.9857 & 0.0143&0.0371			& -3.0350& 0.0350& 0.0583  \\
			1 & 1.0775 & 0.0775 & 0.1030			& 0.9572& 0.0428& 0.2583  \\
			2 &2.0033 &0.0033 & 0.0205			& 2.0452&0.0452& 0.0463 \\ \hline
			0.4751 & 0.4731& 0.0020 & 0.0092			& 0.4719& 0.0032& 0.0321 \\
			-0.1622 & -0.1620 & 0.0002& 0.0059			& -0.1632& 0.0010& 0.0197  \\
			0.3708 & 0.3708 &0& 0.0037			&0.3731&  0.0023& 0.0167 \\ \hline \hline
			\multicolumn{7}{|c|}{$n_3=5000$} \\\hline
			&  \multicolumn{3}{|c||}{Whittle} & \multicolumn{3}{|c|}{QMLE}\\ \hline
			\hspace*{0.1cm} $\vartheta_0$\hspace*{0.1cm} &mean  & bias & std.  & mean  & bias &std.   \\ \hline	
			-1 &-1.0028& 0.0028& 0.0172& -0.9960& 0.0040&0.0174 \\
			-2 & -1.9954& 0.0146& 0.0041& -2.0059& 0.0059&0.0196 \\
			1 &0.9972& 0.0028& 0.0133& 1.0052& 0.0052& 0.0268 \\
			-2 &-2.0202& 0.0202& 0.0210&-2.0043&  0.0043& 0.0284\\
			-3 &-3.0091&0.0091& 0.0441& -3.0013& 0.0013&0.0261  \\
			1 &1.0585&0.0585&0.0409&1.0253&0.0253&0.1249  \\
			2 &2.0109&0.0109& 0.0318&2.0479 &0.0479&0.0346 \\ \hline
			0.4751 &0.4759&0.0008&0.0100&0.4735&0.0016&0.0200  \\
			-0.1622 &-0.1652&0.0030&0.0088&-0.1634&0.0012&0.0135  \\
			0.3708 &0.3904&0.0196&0.0079&0.3727 &0.0019&0.0109 \\ \hline
			
		\end{tabular}
	\end{center}
	\begin{center}
	\caption{\label{table_1}	Estimation results for a  Brownian motion driven bivariate MCARMA(2,1) process with parameter $\vartheta_0^{(1)}$.}
	\end{center}
	
\end{table}

\begin{table}[h]	
	\begin{center}
		\begin{tabular}{|c||c|c|c||c|c|c|}\hline
			\multicolumn{7}{|c|}{$n_1=500$} \\\hline\hline
			&  \multicolumn{3}{|c||}{Whittle} & \multicolumn{3}{|c|}{QMLE}\\ \hline
			\hspace*{0.1cm} $\vartheta_0$\hspace*{0.1cm} &mean & bias & std.  & mean  & bias &std.   \\ \hline
			-1 &-0.9555&0.0445&0.1559 &-0.9651& 0.0349& 0.1854 \\
			-2 &-1.8822&0.1178&0.2653 &-1.6978&0.3022&0.3452 \\
			1 & 0.8746&0.1254&0.1888& 1.1479& 0.1479& 0.2526 \\
			-2 &-2.0981&0.0981&0.2273 &-2.0066&0.0066&0.2962\\
			-3 &-3.1833&0.1833&0.2517 & -3.0578&0.0578&0.4076  \\
			1 &1.0533&0.0533&0.3614 &1.0272& 0.0272& 1.2301  \\
			2 &2.0461&0.0461&0.5710 &2.0490&0.0490&1.6673 \\ \hline
			0.4751 &0.4992&0.0241&0.1061 &0.4645& 0.0106&0.8220  \\
			-0.1622 & -0.1520&0.0102&0.1130 &-0.1669&0.0047&0.3317  \\
			0.3708&0.4100&0.0392&0.1081 &0.3748&0.0040&0.6100 \\ \hline\hline
			\multicolumn{7}{|c|}{$n_2=2000$} \\\hline
			&  \multicolumn{3}{|c||}{Whittle} & \multicolumn{3}{|c|}{QMLE}\\ \hline
			\hspace*{0.1cm} $\vartheta_0$\hspace*{0.1cm} &mean  & bias & std.  & mean  & bias &std.   \\ \hline
			-1 &-1.0351&0.0351&0.1224 & -0.9673&0.0327& 0.0243\\
			-2 &-1.8779&0.1221&0.1894 &-1.0564& 0.0426& 0.0713\\
			1 &0.9457&0.0543&0.2620 &1.1331& 0.1331& 0.1214\\
			-2 &-1.9586&0.0414&0.2573 &-1.9494& 0.0506& 0.0827\\
			-3 &-3.1682&0.1682&0.2238 & -3.1990& 0.1990& 0.4911\\
			1 &1.1234&0.1234&0.3120 & 1.1720& 0.1720& 0.5933\\
			2 &2.0842&0.0842&0.4842 & 2.0432& 0.0432& 0.1817\\ \hline
			0.4751 &0.5010&0.0259&0.1000 &  0.5237& 0.0486& 0.2726\\
			-0.1622 &-0.1740&0.0118&0.0992 & -0.0856& 0.0766& 0.1413\\
			0.3708 & 0.3908&0.0200&0.0758& 0.3220& 0.0488& 0.0049\\ \hline\hline
			
			\multicolumn{7}{|c|}{$n_3=5000$} \\\hline
			&  \multicolumn{3}{|c||}{Whittle} & \multicolumn{3}{|c|}{QMLE}\\ \hline
			\hspace*{0.1cm} $\vartheta_0$\hspace*{0.1cm} &mean & bias & std.  & mean  & bias &std.   \\ \hline
			-1 &-1.0238& 0.0238& 0.1182& -0.9844& 0.0156&0.0194 \\
			-2 & -1.9954& 0.0046& 0.2048& -2.0139& 0.0139&0.0246 \\
			1 &0.9942& 0.0058& 0.1517& 1.0102& 0.0102& 0.0299 \\
			-2 &-2.2202& 0.2202& 0.2210&-2.0043&  0.0043& 0.0284\\
			-3 &-3.0104&0.0104& 0.2463& -3.0015& 0.0015&0.2291  \\
			1 &1.0585&0.0585&0.2409&1.0655&0.0655&0.1347  \\
			2 &2.1169&0.1169& 0.0866&2.0400 &0.0400&0.0355 \\ \hline
			0.4751 &0.4855&0.0104&0.1180&0.4737&0.0018&0.0206  \\
			-0.1622 &-0.1682&0.0060&0.0408&-0.1634&0.0012&0.0145  \\
			0.3708 &0.3908&0.0200&0.0842&0.3730 &0.0022&0.0139 \\ \hline

		\end{tabular}
	\end{center}
	
	\begin{center}
	 \caption{	\label{table_2}	Estimation results for a  NIG driven bivariate MCARMA(2,1) process with parameter $\vartheta_0^{(1)}$.}
	\end{center}
	
\end{table}


Since we introduced an alternative estimator for the univariate setting, we perform an additional simulation study concerning one dimensional CARMA processes. In accordance to Assumption $\widetilde A$, the variance parameter $\sigma_L^2$ of the Lévy process is fixed in this study and has not to be estimated.
We consider a CARMA(2,1) model where
\begin{align*}
A(\vartheta)&=\left( \begin{array}{c c}0 & 1 \\  \vartheta_1 &\vartheta_2\end{array}\right),\quad \quad B(\vartheta)=\left( \begin{array}{c}\vartheta_3\\\vartheta_1+\vartheta_2\vartheta_3 \end{array}\right)\quad \text{ and } \quad C(\vartheta)=(1 \ 0 ).
\end{align*}
Since the output process $Y(\vartheta)$ of this minimal state space model is of dimension one, the order of the AR polynomial $p$ is equal to $N=2$
and the order of the MA polynomial is $q=p-1=1$. This means we have a CARMA$(2,1)$ process.
For more details on CARMA processes we refer to \cite{Brockwell:Lindner:2009,Brockwell:2014}.
In our simulation study
the true parameter is
$$\vartheta_0^{(2)}=(-2, -2 , -1 ).$$

The simulation results for the Brownian motion driven and the NIG driven CARMA(2,1) process are given in ~\Cref{table_3} and \Cref{table_4}, respectively.
For all  sample sizes, the Whittle estimator and the  QMLE behave very similar and give excellent estimation results. Whereas for small sample sizes the adjusted Whittle estimator is remarkably worse, for increasing sample sizes it performs much better and seems to converge.
Further simulations for a bivariate  MCAR(1) process and an univariate CAR(3) process showing a similar pattern as the simulations
of this section are presented in \Cref{se:further simulation} in the Supplementary Material.

\begin{table}[h]	
	\begin{center}
		\begin{tabular}{|c||c|c|c||c|c|c||c|c|c|}\hline
			\multicolumn{10}{|c|}{$n_1=500$} \\\hline
			&  \multicolumn{3}{|c||}{Whittle} &\multicolumn{3}{|c||}{adjusted Whittle}& \multicolumn{3}{|c|}{QMLE}\\ \hline
			\  $\vartheta_0$\hspace*{0.05cm} &mean & bias& std. & mean   & bias  &std.  & mean  & bias &std.  \\ \hline
			-2& -2.0951& 0.0951& 0.7766& 3.1063& 1.1063& 3.4195& -2.0880& 0.0880& 0.7628\\
			-2& -2.0482& 0.0482& 0.6500 &-2.9233& 0.9233& 2.9957& -2.0449& 0.0449& 0.5889\\
			-1& -0.9731& 0.0269& 0.1186& -0.9028& 0.0972& 0.3710& -0.9729& 0.0271& 0.1779\\
			\hline \hline
			\multicolumn{10}{|c|}{$n_2=2000$} \\\hline
			&  \multicolumn{3}{|c||}{Whittle} &\multicolumn{3}{|c||}{adjusted Whittle}& \multicolumn{3}{|c|}{QMLE}\\ \hline
			\ $\vartheta_0$\hspace*{0.05cm} &mean  & bias & std.& mean   & bias  &std.   & mean & bias&std.\\ \hline
			-2& -2.0204& 0.0204& 0.0755&-2.0816& 0.0816& 1.0399 & -2.0015& 0.0015& 0.1926\\
			-2& -1.9975& 0.0025& 0.0637& -2.0732& 0.0732& 0.9199 & -1.9948& 0.0052& 0.1466\\
			-1& -0.9933& 0.0067& 0.0547& -0.9965& 0.0035& 0.1267& -0.9993& 0.0007& 0.0674 \\
			\hline \hline
			\multicolumn{10}{|c|}{$n_3=5000$} \\\hline \hline
			&  \multicolumn{3}{|c||}{Whittle} &\multicolumn{3}{|c||}{adjusted Whittle}& \multicolumn{3}{|c|}{QMLE}\\ \hline
			\ $\vartheta_0$\hspace*{0.05cm} &mean & bias & std. & mean   & bias  &std.   & mean  & bias &std. \\ \hline	
			-2& -2.0046& 0.0046& 0.0117 &-1.9854& 0.0146& 0.0860& -2.0068& 0.0068&0.0997\\
			-2& -1.9914& 0.0086& 0.0149& -1.9840& 0.0160& 0..0821& -1.9942& 0.0058& 0.0772\\
			-1& -1.0004& 0.0004& 0.0153& -1.0070& 0.0070& 0.0488& -1.0009& 0.0009& 0.0408\\
			\hline
			
		\end{tabular}
	\end{center}
	\caption{\label{table_3} 	Estimation results for a Brownian motion driven CARMA(2,1) process with parameter $\vartheta_0^{(2)}$.}
\end{table}

\begin{table}[h]	
	\begin{center}
		\begin{tabular}{|c||c|c|c||c|c|c||c|c|c|}\hline
			\multicolumn{10}{|c|}{$n_1=500$} \\\hline
			&  \multicolumn{3}{|c||}{Whittle} &\multicolumn{3}{|c||}{adjusted Whittle}& \multicolumn{3}{|c|}{QMLE}\\ \hline
			\ $\vartheta_0$\hspace*{0.1cm} &mean  & bias & std. & mean   & bias  &std.  & mean & bias &std.  \\ \hline
			-2& -2.3278& 0.3278& 1.7598& -3.0174 & 1.0174& 3.2090&-2.3175& 0.3175& 1.0862\\
			-2& -2.2612& 0.2612& 1.4892& -2.8550& 0.8550& 2.8684& -2.2047& 0.2047& 0.8023 \\
			-1& -0.9855& 0.0145& 0.1652& -0.9445& 0.0555& 0.3376& -0.9243& 0.0757& 0.2938 \\
			\hline \hline
			\multicolumn{10}{|c|}{$n_2=2000$} \\\hline
			&  \multicolumn{3}{|c||}{Whittle} &\multicolumn{3}{|c||}{adjusted Whittle}& \multicolumn{3}{|c|}{QMLE}\\ \hline
			\ $\vartheta_0$\hspace*{0.1cm} &mean  & bias & std. & mean   & bias  &std.   & mean  & bias &std. \\ \hline
			-2& -2.0261& 0.0261& 0.1038& -1.9996& 0.0004& 0.5351& -2.0122& 0.0122& 0.2526 \\
			-2& -1.9977& 0.0023& 0.0784& -1.9988& 0.0012& 0.4552 & -2.0034& 0.0034& 0.1845\\
			-1& -0.9968& 0.0032& 0.0607& -1.0153& 0.0153& 0.0961& -1.0037& 0.0037& 0.0848
			\\ \hline \hline
			\multicolumn{10}{|c|}{$n_3=5000$} \\\hline \hline			&  \multicolumn{3}{|c||}{Whittle} &\multicolumn{3}{|c||}{adjusted Whittle}& \multicolumn{3}{|c|}{QMLE}\\ \hline
			\  $\vartheta_0$\hspace*{0.1cm} &mean  & bias & std. & mean   & bias  &std.   & mean  & bias &std.  \\ \hline	
			-2& -2.0138& 0.0138 & 0.0575& -1.9842& 0.0158& 0.0902 & -1.9938& 0.0062& 0.1093\\
			-2& -1.9948& 0.0052& 0.0466 & -1.9866& 0.0134& 0.0825& -1.9917& 0.0083& 0.0906\\
			-1& -0.9991& 0.0009& 0.0339& -1.0097& 0.0097& 0.0508& -1.0059& 0.0059& 0.0415 \\
			\hline
			
		\end{tabular}
	\end{center}
	\caption{	\label{table_4} 	Estimation results for a NIG driven CARMA(2,1) process with parameter $\vartheta_0^{(2)}$.}
\end{table}
\FloatBarrier

\section{Proofs  for the Whittle estimator in ~\Cref{sec:WhittleEstimator}} \label{sec:Proofs:WhittleEstimator}

\subsection{Proofs of \Cref{sec:Cons:Whittle}} \label{sec:Proof:Theorem3.2}

\noindent\textit{Proof of \Cref{Hilfssatz1}.}
	We divide $W_n$ in two parts and investigate them separately.
	Therefore, define  $$W_n^{(1)}(\vartheta):=\frac{1}{2n}\sum_{j=-n+1}^{n}\text{tr}\left(\f(\omega_j, \vartheta)^{-1}I_n(\omega_j)\right)$$
	and  $$W_n^{(2)}(\vartheta)=\frac{1}{2n}\sum_{j=-n+1}^{n}\log\left(\det\left(\f(\omega_j,\vartheta)\right)\right),$$ such that $W_n(\vartheta)=W_n^{(1)}(\vartheta)+W_n^{(2)}(\vartheta).$ Since $(A1)$ and $(A4)$ are satisfied, we can apply \Cref{unifdet} of the Supplementary Material, which gives the uniform convergence \begin{align}
	\sup_{\vartheta \in \Theta}\left|W_n^{(2)}(\vartheta)-\frac 1 {2\pi}\int_{-\pi}^{\pi}\log\left(\det\left(\f(\omega,\vartheta)\right)\right)d\omega\right|\overset{n\to \infty}{\longrightarrow}0.
	\end{align}
	It remains to prove the appropriate convergence of $W_n^{(1)}$. Therefore, it is sufficient to show that
	\begin{align}\label{glmKW1}\sup_{\vartheta \in \Theta}\left\|\frac{1}{2n}\sum_{j=-n+1}^{n}\f(\omega_j, \vartheta)^{-1}I_n(\omega_j)-\frac{1}{2\pi}\int_{-\pi}^{\pi}\f(\omega,\vartheta)^{-1}{\f(\omega)}d\omega\right\|\overset{a.s.}{\longrightarrow}0\end{align} holds.
	We approximate $\f(\omega_j, \vartheta)^{-1}$ by the Cesàro sum of its Fourier series of size $M$  for $M$ sufficiently large. Define
	\begin{align*}
	q_M(\omega,\vartheta)&:=\frac{1}{M}\sum_{j=0}^{M-1}\left(\sum_{|k|\leq j}{b}_k(\vartheta) e^{-ik\omega} \right)=\sum_{|k|<M}\left(1-\frac{|k|}{M}\right)b_k(\vartheta) e^{-ik\omega} \quad \text{ with }\\b_k(\vartheta)&:= \frac 1 {2\pi}\int_{-\pi}^{\pi} f_Y^{(\Delta)}(\omega,\vartheta)^{-1}e^{ik\omega}d\omega .
	\end{align*}
	The inverse $f_Y^{(\Delta)}(\omega,\vartheta)^{-1}$ exists, is continuous and
	$2\pi$-periodic in the first component. Thus, an application of  \Cref{Fejer} of the Supplementary Material gives that for any $\epsilon>0$ there exists an $M_0(\epsilon)\in \N$ such that for $M\geq M_0(\epsilon)$  \begin{align}\label{H1}
	\sup_{ \omega \in[-\pi,\pi]}\sup_{\vartheta \in \Theta}\left\|\f(\omega,\vartheta)^{-1}-q_M(\omega,\vartheta)\right\|<\epsilon.
	\end{align} Let $\epsilon>0.$
	In view of \eqref{H1}, we get \begin{eqnarray}\label{hh3}\left\|\frac{1}{2n}\sum_{j=-n+1}^{n}\f(\omega_j, \vartheta)^{-1}I_n(\omega_j)-\frac{1}{2n}\sum_{j=-n+1}^{n}q_M(\omega_j,\vartheta)I_n(\omega_j) \right\|\leq\frac{\epsilon}{2n}\sum_{j=-n+1}^{n}\left\|I_n(\omega_j)\right\|.
	\end{eqnarray}
	Since all matrix norms are equivalent, using the 1-norm yields
	\begin{align}\label{hh1}
	\frac{\epsilon}{2n}\sum_{j=-n+1}^{n}\left\|I_n(\omega_j)\right\|\leq \frac{\epsilon \mathfrak C}{2n}\sum_{j=-n+1}^{n}\sum_{k=1}^m \sum_{\ell=1}^{m}|I_n(\omega_j)[{k,\ell}]|.
	\end{align}
	The representation \eqref{Per_ACVF} of the periodogram and the non-negativeness of any one dimensional periodogram imply that ${a^\top I_n(\omega_j)a=I_{n,a^\top Y}(\omega_j)\geq 0}$ so that $I_n(\omega_j)$ is a positive semi-definite and Hermitian matrix. Therefore,  for $k,\ell \in \{1,\ldots,m\},\ j \in \{-n+1,\ldots,n\},$  \begin{align*}
	\det\begin{pmatrix}
	I_n(\omega_j)[{k,k}] & I_n(\omega_j)[{k,\ell}] \\
	I_n(\omega_j)[{\ell,k}] &  I_n(\omega_j)[{\ell,\ell}]
	\end{pmatrix}\geq 0,\end{align*} which implies \begin{align}\label{hh2}\left|I_n(\omega_j)[{k,\ell}]\right|\leq \sqrt{I_n(\omega_j)[{k,k}]I_n(\omega_j)[{\ell,\ell}]}\leq I_n(\omega_j)[{k,k}]+I_n(\omega_j)[{\ell,\ell}].\end{align}
	Combining \eqref{hh3}, \eqref{hh1}, \eqref{hh2} and \Cref{summee} of the Supplementary Material gives for $M\geq M_0(\epsilon)$
	\begin{eqnarray*}
		\lefteqn{\left\|\frac{1}{2n}\sum_{j=-n+1}^{n}\f(\omega_j, \vartheta)^{-1}I_n(\omega_j)-\frac{1}{2n}\sum_{j=-n+1}^{n}q_M(\omega_j,\vartheta)I_n(\omega_j) \right\|}\\
		&&\leq \frac{\epsilon \mathfrak C}{2n}\sum_{j=-n+1}^{n}\sum_{k=1}^m \sum_{\ell=1}^{m}\left[I_n(\omega_j)[{k,k}]+I_n(\omega_j)[\ell,\ell]\right]\\
		&&\leq\frac{\epsilon \mathfrak Cm}{n}\sum_{j=-n+1}^{n}\sum_{k=1}^m I_n(\omega_j)[{k,k}]\\
		&&\leq {2\epsilon\mathfrak Cm} \sum_{k=1}^{m}\overline \Gamma^{(\Delta)}_n(0)[{k,k}].
	\end{eqnarray*}
	Since
	$ \sum_{k=1}^{m}\overline \Gamma_n^{(\Delta)}(0)[{k,k}]\overset{a.s.}{\longrightarrow}\sum_{k=1}^{m} \Gamma^{(\Delta)}(0)[{k,k}]<\infty$ due to \Cref{ewertmpAC} in the Supplementary Material,
	we obtain for $M\geq M_0(\epsilon)$ and $n$ large
    $$\sup_{\vartheta \in \Theta}\left\|\frac{1}{2n}\sum_{j=-n+1}^{n}\left(\f(\omega_j, \vartheta)^{-1}I_n(\omega_j)\right)-\frac{1}{2n}\sum_{j=-n+1}^{n}q_M(\omega_j,\vartheta)I_n(\omega_j) \right\|\leq\varepsilon\mathfrak C$$ almost surely.  Consequently, for the proof of \eqref{glmKW1} it is sufficient to show that \begin{align}\label{glmKW2}
	\sup_{\vartheta\in \Theta}\left\|\frac 1 {2n}\sum_{j=-n+1}^{n}q_M(\omega_j,\vartheta)I_n(\omega_j)-\frac{1}{2\pi}\int_{-\pi}^{\pi}\f(\omega,\vartheta)^{-1}\f(\omega)d\omega \right\| \overset{a.s.}\longrightarrow 0.
	\end{align}
	On the one hand, \Cref{summee} of the Supplementary Material yields
	\begin{eqnarray} \nonumber\frac{1}{2n}\sum_{j=-n+1}^{n}q_M(\omega_j,\vartheta)I_n(\omega_j)&=&\frac{1}{2\pi}\sum_{|k|<M}\sum_{|h|<n}\left(\left(1-\frac{|k|}{M}\right)b_k(\vartheta)\overline\Gamma_n^{(\Delta)}(h){\left(\frac{1}{2n}\sum_{j=-n+1}^{n}e^{-i(k+h)\omega_j}\right)}\right)\\
	\nonumber&=&\frac{1}{2\pi}\sum_{|k|<M}\left(1-\frac{|k|}{M}\right)b_k(\vartheta)\overline\Gamma_n^{(\Delta)}(-k)\\
	\label{2.1.a}&\overset{a.s.}{\longrightarrow}&\frac{1}{2\pi}\sum_{|k|<M}\left(1-\frac{|k|}{M}\right)b_{k}(\vartheta)\Gamma^{(\Delta)}(-k)
	\end{eqnarray}
	uniformly in $\vartheta$, since $b_{k}(\vartheta)$ is uniformly bounded in $\vartheta$ for all $k$. The reason is that $f_Y^{(\Delta)}(\omega,\vartheta)^{-1}$ is continuous on the compact set $[-\pi,\pi]\times \Theta$ and  $$\sup_{\substack{\vartheta\in \Theta\\ k \in \Z}} \|b_k(\vartheta)\|=\sup_{\substack{\vartheta\in \Theta\\ k \in \Z}} \left\| \frac 1 {2\pi}\int_{-\pi}^{\pi} f_Y^{(\Delta)}(\omega,\vartheta)^{-1}e^{ik\omega}d\omega \right\|\leq \max_{\vartheta\in \Theta} \max_{\omega\in [-\pi,\pi]}\|f_Y^{(\Delta)}(\omega,\vartheta)^{-1} \|.$$
	On the other hand, due to \eqref{kovarianz}, we get
	\begin{eqnarray} \nonumber\lefteqn{\left\| \frac{1}{2\pi}\sum_{|h|<M}\left(1-\frac{|h|}{M}\right)b_{-h}(\vartheta)\Gamma^{(\Delta)}(h) -\frac{1}{2\pi}\int_{-\pi}^{\pi}\f(\omega,\vartheta)^{-1}\f(\omega)d\omega\right\|}\\ \nonumber&\overset{}{=}&\left\| \frac{1}{2\pi}\sum_{|h|<M}\left(1-\frac{|h|}{M}\right)b_{-h}(\vartheta)\int_{-\pi}^{\pi}\f(\omega)e^{ih\omega}d\omega -\frac{1}{2\pi}\int_{-\pi}^{\pi}\f(\omega,\vartheta)^{-1}\f(\omega)d\omega\right\|\\
	\nonumber&=&\left\|\frac{1}{2\pi}\int_{-\pi}^{\pi}\left(q_M(\omega,\vartheta)-\f(\omega,\vartheta)^{-1}\right)\f(\omega)d\omega \right\| \\
	\label{2.1b}	&\ \leq&\frac{1}{2\pi}\int_{-\pi}^{\pi}\left\|q_M(\omega,\vartheta)-\f(\omega,\vartheta)^{-1}\right\|\left\|\f(\omega)\right\|d\omega \leq \epsilon \mathfrak{C},
	\end{eqnarray}
	where we used \eqref{H1} and the continuity of $\f(\omega)$ for the last inequality. Combining \eqref{2.1.a} and \eqref{2.1b} gives \eqref{glmKW2}.
\hfill$\Box$\\[1mm]

\noindent\textit{Proof of \Cref{schritt2}.}
	In view of \Cref{invertible}, we express the linear innovations as
	\begin{align*}
	\varepsilon_{k}^{(\Delta)}(\vartheta)=\Pi(\mathsf{B},\vartheta)Y_k^{(\Delta)}(\vartheta), \quad k\in\N,
	\end{align*}
	and define the pseudo innovations as
	\begin{align*}
	\xi_{k}^{(\Delta)}(\vartheta):=\Pi(\mathsf{B},\vartheta)Y_k^{(\Delta)}(\vartheta_0), \quad k\in\N.
	\end{align*}
	An application of Theorem 11.8.3 of \cite{BrockwellDavis}  leads to the spectral densities of $(\varepsilon_k^{(\Delta)}(\vartheta))_{k\in\N}$ and $(\xi^{(\Delta)}_k(\vartheta))_{k\in\N}$ as
	\begin{align*}
	f^{(\Delta)}_\varepsilon(\omega,\vartheta)&=\Pi(e^{-i\omega},\vartheta)\f(\omega,\vartheta)\Pi(e^{i\omega},\vartheta)^\top,\quad \omega\in[-\pi,\pi],\\
	f^{(\Delta)}_\xi(\omega,\vartheta)&=\Pi(e^{-i\omega},\vartheta)\f(\omega)\Pi(e^{i\omega},\vartheta)^\top,\quad \omega\in[-\pi,\pi],
	\end{align*}
	respectively.
	Consequently,
	\begin{eqnarray*}
		\lefteqn{\frac{1}{2\pi}\int_{-\pi}^{\pi}\text{tr}\left(\f(\omega,\vartheta)^{-1}{\f(\omega)}\right)d\omega}\\
		&\quad\quad\overset{}{=}&\frac{1}{2\pi}\text{tr}\left(\int_{-\pi}^{\pi}2\pi\Pi(e^{i\omega},\vartheta)^\top{V^{(\Delta)}(\vartheta)^{-1}}\Pi(e^{-i\omega},\vartheta)\f(\omega)d\omega\right)\\
		&\quad\quad{=}&\text{tr}\left(V^{(\Delta)}(\vartheta)^{-1}\int_{-\pi}^{\pi}f^{(\Delta)}_\xi(\omega,\vartheta)d\omega\right)\\
		&\quad\quad\overset{}{=}&\E\left[\text{tr}\left(\xi^{(\Delta)}_1(\vartheta)^\top V^{(\Delta)}(\vartheta)^{-1}\xi^{(\Delta)}_{1}(\vartheta)\right)\right]
	\end{eqnarray*}
	holds.
	Finally,
	$$\frac{1}{2\pi}\int_{-\pi}^{\pi}\log\left(\det\left(\f(\omega,\vartheta)\right)\right)d\omega=\frac{1}{2\pi}\int_{-\pi}^{\pi}\log\left(\det\left(2\pi \f(\omega,\vartheta)\right)\right)d\omega- m\log(2\pi),$$ and an application of Theorem 3'{}'{}' of Chapter 3 of \cite{Hannanmult} results in \begin{align}\label{V3}\frac{1}{2\pi}\int_{-\pi}^{\pi}\log\left(\det\left(2\pi \f(\omega,\vartheta)\right)\right)d\omega- m\log(2\pi)=\log (\det V^{(\Delta)}(\vartheta))-m\log(2\pi),\end{align}	which completes the proof.
\hfill$\Box$\\[1mm]

\noindent\textit{Proof of \Cref{schritt2.2}.}
	Considering \Cref{schritt2} we get  $W(\vartheta)= \mathcal{L}(\vartheta)$. \cite{QMLE}, Lemma 2.10, proved that $\mathcal{L}$ has a unique global minimum in $\vartheta_0$ under conditions which are fulfilled in our setting (see Lemma 2.3 and Lemma 3.14 of \cite{QMLE}).
\hfill$\Box$\\[1mm]

\noindent \textit{Proof of \Cref{satz1}.}
Due to \Cref{Hilfssatz1} and \Cref{schritt2.2}, we know that the Whittle function $W_n$ converges almost surely uniformly to $W$ and that $W$ has a unique global minimum in $\vartheta_0$. It remains to show that the minimizing arguments of $W_n$ converge almost surely to the minimizer of $W$.
To that effect, we first prove \begin{align}\label{schritt3zz}W_n(\widehat\vartheta_n^{(\Delta)})\overset{a.s.}{\longrightarrow}W(\vartheta_0)\end{align} and deduce that for every neighborhood $U$ of $\vartheta_0$ Whittle`s estimate $\widehat\vartheta_n^{(\Delta)}$ lies in $U$ almost surely for $n$ large enough.

In view of \Cref{Hilfssatz1}, for all $\epsilon>0$ there exists some $n_0 \in \N$ with \begin{align}\label{con1} \sup_{\vartheta \in \Theta} |W_n(\vartheta)-W(\vartheta)|\leq \epsilon \quad \forall\ n\geq n_0\quad \P\text{-a.s.}\end{align}
Therefore, using the definition of $\widehat\vartheta_n^{(\Delta)}$ and \Cref{schritt2.2}, we get for $n\geq n_0$
\begin{align*}
W_n(\widehat\vartheta_n^{(\Delta)})&\leq W_n(\vartheta_0)\leq W(\vartheta_0)+\epsilon \quad \P\text{-a.s.} \quad \text{ and }\\
W_n(\widehat\vartheta_n^{(\Delta)})&\geq W(\widehat\vartheta_n^{(\Delta)})-\epsilon\geq W(\vartheta_0)-\epsilon \quad \P\text{-a.s.}
\end{align*}
and hence,$$\sup_{n\geq n_0}|W_n(\widehat\vartheta_n^{(\Delta)})-W(\vartheta_0)|\leq \epsilon \quad \P\text{-a.s.}$$
follows. This gives the desired convergence \eqref{schritt3zz}.
Now, define $\delta(U):=\inf_{\vartheta\in \Theta\setminus U}W(\vartheta)-W(\vartheta_0)>0$ for any neighborhood of $U$ of $\vartheta_0$. The inequalities \begin{eqnarray*}
	\lefteqn{\P\left(\lim_{n\to \infty} \widehat\vartheta_n^{(\Delta)}=\vartheta_0\right)=\P\left(\forall\ U\ \exists\  n_0(U) \in \N:\ \widehat\vartheta_n^{(\Delta)}\in U\  \forall\ n\geq n_0(U)\right)}\\
	&\geq& \P\left(\forall\ U \ \exists\ n_0(U)\in \N:\ |W_n(\widehat\vartheta_n^{(\Delta)})-W(\vartheta_0)|<\frac{\delta(U)}{2}\right.\\
	&&\left. \quad \quad \text{ and } |W_n(\widehat\vartheta_n^{(\Delta)})-W(\widehat\vartheta_n^{(\Delta)})|<\frac{\delta(U)}{2}\  \forall\ n\geq n_0(U)\right)=1,
\end{eqnarray*}
where the last equality follows from \eqref{schritt3zz} and \Cref{Hilfssatz1}, complete the proof.
\hfill$\Box$

\subsection{Proofs of \Cref{sec:Whittle:asymptoticNormal}} 


\noindent\textit{Proof of \Cref{remarkunif}.}
	Under the Assumptions $(A1)$--$(A4)$ and $(B3)$ the spectral density $\f(\omega,\vartheta)$ and its inverse $\f(\omega,\vartheta)^{-1}$ are three times continuously differentiable in $\vartheta$ (see \Cref{identifizierbar} and \Cref{BemAsyNorm}). Furthermore,  $$\frac{\partial^2}{\partial \vartheta_k \partial \vartheta_\ell} \tr \left(\f(\omega,\vartheta)^{-1}I_n(\omega)\right)= \tr\left(\frac{\partial^2}{\partial \vartheta_k \partial \vartheta_\ell}\left(\f(\omega,\vartheta)^{-1}\right)I_n(\omega)\right),\quad  k,\ell \in \{1,\ldots, r\}.$$ Therefore, the proof of $$\sup_{\vartheta \in \Theta}\left\|\nabla^2_\vartheta W_n(\vartheta)- \nabla^2_\vartheta W(\vartheta)\right\|\overset{a.s.}{\longrightarrow}0$$  goes in the same way as the proof of \Cref{Hilfssatz1}.
	It remains to show that $\nabla^2_\vartheta W(\vartheta_0)=\Sigma_{\nabla^2 W}$. 
	
	First, note that \begin{align}\label{zweiteAblDarst1} \nabla^2_\vartheta W(\vartheta_0)=\frac{1}{2\pi}\int_{-\pi}^{\pi} \nabla_\vartheta^2 \tr(\f(\omega,\vartheta_0)^{-1}\f(\omega))+\nabla_\vartheta^2 \log\left(\det \left(\f(\omega,\vartheta_0)\right)\right)d\omega.\end{align}
	On the one hand, \begin{eqnarray}
	\nonumber\lefteqn{\frac{1}{2\pi}\int_{-\pi}^{\pi}\tr\left(\frac {\partial^2}{\partial \vartheta_k \partial \vartheta_l}\left(\f(\omega,\vartheta_0)^{-1}\right)\f(\omega)\right)d\omega}\\
	\nonumber&&=\frac{1}{2\pi}\int_{-\pi}^{\pi}\tr\left(2\f(\omega)^{-1}\left(\frac {\partial}{\partial \vartheta_k}\f(\omega,\vartheta_0)\right)\f(\omega)^{-1}\left(\frac{\partial}{\partial \vartheta_\ell}\f(\omega,\vartheta_0)\right)\right.\\&&\label{Vickys2}\quad\quad\quad\quad\quad\left.-\f(\omega)^{-1}\left(\frac{\partial^2}{\partial \vartheta_k\partial \vartheta_\ell}\f(\omega,\vartheta_0)\right)\right)d\omega
	\end{eqnarray}
	holds.
	On the other hand, Jacobi's formula leads to
	\begin{eqnarray}\nonumber
	\lefteqn{\frac{\partial^2}{\partial \vartheta_k\partial \vartheta_\ell}\log(\det(\f(\omega,\vartheta_0)))}\\
	&&=\tr\left(-\f(\omega)^{-1}\left(\frac {\partial}{\partial \vartheta_k}\f(\omega,\vartheta_0)\right)\f(\omega)^{-1}\left(\frac{\partial}{\partial \vartheta_\ell}\f(\omega,\vartheta_0)\right)\right)\nonumber\\
	\label{Vickys3}&&\quad+\tr\left(\f(\omega)^{-1}\left(\frac{\partial^2}{\partial \vartheta_k\partial \vartheta_\ell}\f(\omega,\vartheta_0)\right)\right).\end{eqnarray}
	Combining \eqref{zweiteAblDarst1}, \eqref{Vickys2}, \eqref{Vickys3} and the property
    \begin{align}\label{vectrace} \vecc\left(A^\top\right)^\top\left(B^\top\otimes C\right) \vecc(D)=\tr\left(BACD\right)
	\end{align}
	for appropriate matrices $A,B,C,D$ (see \cite{Brewer}, properties T2.4, T3.4 and T3.8) gives
	\begin{align*}
	{\nabla^2_\vartheta W(\vartheta_0)}&=\frac{1}{2\pi}\int_{-\pi}^{\pi}\nabla_\vartheta \f(-\omega,\vartheta_0)^{\top}\left[\f(-\omega)^{-1}\otimes \f(\omega)^{-1}\right]\nabla_\vartheta\f(\omega,\vartheta_0)d\omega=\Sigma_{\nabla^2W}.
	\end{align*}
\hfill$\Box$\\[1mm]

\noindent\textit{Proof of \Cref{positiveSigmanabla2}.}
	Let $c \in \C^r$ be fixed and $\omega^*$ as in $(B4)$. The continuity of $\f(\omega)$ and its regularity imply for any $\omega$ in a neighborhood of
	$\omega^*$ that
	$$\left\|\left(\f(-\omega)^{-1/2}\otimes\f(\omega)^{-1/2}\right)\nabla_\vartheta\f(\omega,\vartheta_0)c\right\|_2 >0$$
	where $\|\cdot\|_2$ is the Euclidean norm.
	Consequently,   \begin{eqnarray*}c^\top\Sigma_{\nabla^2 W}c&=&\frac{1}{2\pi}\int_{-\pi}^{\pi}c^\top \nabla_\vartheta \f(\omega,\vartheta_0)^{H}\left[\f(-\omega)^{-1}\otimes \f(\omega)^{-1}\right]\nabla_\vartheta\f(\omega,\vartheta_0) c d\omega\\
		&=&\frac{1}{2\pi}\int_{-\pi}^{\pi}\left\|\left(\f(-\omega)^{-1/2}\otimes\f(\omega)^{-1/2}\right)\nabla_\vartheta\f(\omega,\vartheta_0)c\right\|_2^2d\omega
		>0.
	\end{eqnarray*}
	Therefore, $\Sigma_{\nabla^2W}$ is positive definite.
\hfill$\Box$\\[1mm]

For the proof of \Cref{Prop2} we require some auxiliary result.
Therefore, we denote the periodogram  and the sample covariance corresponding to $N_1^{(\Delta)},\ldots,N_n^{(\Delta)}$ as $I_{n,N}$ and $\overline\Gamma_{n,N}$, respectively.

\begin{lemma}\label{Prop1}
	Let Assumptions $(A2)$--$(A4)$ hold and
	$\eta:[-\pi,\pi]\to \C^{m\times m}$ be a symmetric matrix-valued continuous function with Fourier coefficients $(\mathfrak{f}_u)_{u\in \Z}$ satisfying $\sum_{u=-\infty}^{\infty} \|\mathfrak{f}_u\|<\infty.$
	Then,
	$$ \lim_{n\to \infty}\E\left|\frac 1{2\sqrt n} \sum_{j=-n+1}^{n}\tr\left(\eta(\omega_j)I_n(\omega_j)-  \eta(\omega_j)\Phi(e^{-i\omega_j})I_{n,N}(\omega_j)\Phi(e^{i\omega_j})^\top\right) \right|=0.$$
\end{lemma}
\begin{proof}
	Define $R_n(\omega)=I_n(\omega)-\Phi(e^{-i\omega})I_{n,N}(\omega) \Phi(e^{i\omega})^\top$ for $\omega \in [-\pi,\pi]$. We get \begin{align*}
	R_n(\omega_j)
	&=\frac{1}{2\pi n}\left(\sum_{k=1}^{n}\sum_{s=0}^{\infty} \Phi_s N_{k-s}\D \right)\left(\sum_{\ell=1}^{n}\sum_{t=0}^{\infty}\Phi_tN_{\ell-t}\D \right)^\top e^{-i(k-\ell)\omega_j}\\
	&\quad -\frac{1}{2\pi n}\left(\sum_{k=1}^{n}\sum_{s=0}^{\infty}\Phi_sN_k\D \right)\left(\sum_{\ell=1}^{n}\sum_{t=0}^{\infty}\Phi_tN_\ell\D \right)^\top e^{-i(k+s-\ell-t) \omega_j} \\
	&=\frac{1}{2\pi n}\left(\sum_{s=0}^{\infty}\sum_{t=0}^{\infty} \Phi_s \left( \left( \sum_{k=1}^{n}\sum_{\ell=1-t}^{0} -\sum_{k=1}^{n}\sum_{\ell=n-t+1}^{n}+\sum_{k=1-s}^{0}\sum_{\ell=1}^{n}+\sum_{k=1-s}^{0}\sum_{\ell=1-t}^{0} \right. \right. \right. \\
	& \quad \quad \quad  \ \left. -\sum_{k=1-s}^{0}\sum_{\ell=n-t+1}^{n}- \sum_{k=n-s+1}^{n}\sum_{\ell=1}^{n}-\sum_{k=n-s+1}^{n}\sum_{\ell=1-t}^{0}+\sum_{k=n-s+1}^{n}\sum_{\ell=n-t+1}^{n}\right)\\
	&\quad\quad \quad \quad  \ \ \ \left. \left. N_k\D N_\ell^{(\Delta)\top}e^{-i(k+s-\ell-t)\omega_j}\right)\Phi_t^\top \right)\\
	&=:\sum_{i=1}^{8}R_n^{(i)}(\omega_j).
	\end{align*}
	Thus,
	\beao
	\E\left|\frac{1}{2\sqrt n}\sum_{j=-n+1}^{n}\tr\left( \eta(\omega_j)R_n(\omega_j)\right)\right|&\leq& \sum_{i=1}^{8}\E\left|\frac{1}{2\sqrt n}\sum_{j=-n+1}^{n}\tr\left( \eta(\omega_j)R^{(i)}_n(\omega_j)\right)\right|.
	\eeao
	We have to show that these 8 components converge to zero. Since we can treat each component similarly, we only give the detailed proof for the convergence of the first term.
	
	Due to $\tr(A)\leq \|A\|_1$ for all quadratic matrices $A$, we get an upper bound for the trace of any quadratic matrix. Once again, the equivalence of all matrix norms and $\eta(\omega_j)=\sum_{u=-\infty}^{\infty}\mathfrak f_ue^{-i\omega_j u}$ yield
	\begin{eqnarray*}
		\lefteqn{\E\left|\frac{1}{2\sqrt n}\sum_{j=-n+1}^{n}\tr\left( R_n^{(1)}(\omega_j)\eta(\omega_j)\right)\right|}\\&\quad \quad\leq&\mathfrak C\E\left\|\sum_{j=-n+1}^{n} \frac{1}{\sqrt n} \frac 1 { n}\sum_{s=0}^{\infty}\sum_{t=0}^{\infty} \Phi_s  \sum_{k=1}^{n}\sum_{\ell=1-t}^{0}  N_k\D N_\ell^{(\Delta)\top}\Phi_t^\top \sum_{u=-\infty}^{\infty}\mathfrak{f}_u e^{-i(k+s-\ell-t+u)\omega_j}\right\|\\
		&\quad \quad\leq&\mathfrak C \frac 1 {\sqrt n}\sum_{s=0}^{\infty}\sum_{t=0}^{\infty}\|\Phi_s\|\sum_{k=1}^{n}\sum_{\ell=1-t}^{0}\E\|N_1^{(\Delta)}\|^2 \|\Phi_t\| \sum_{u=-\infty}^{\infty}\|\mathfrak f_u\|\frac 1 n\left\| \sum_{j=-n+1}^{n}e^{-i(k+s-\ell-t+u)\omega_j}\right\|.
	\end{eqnarray*}
	Due to $(A2)$, $\E\|N_1^{(\Delta)}\|^2<\infty.$ 	
	Further, an application of \Cref{summee} of the Supplementary Material gives
	\begin{eqnarray*}
		\E\left|\frac{1}{2\sqrt n}\sum_{j=-n+1}^{n}\tr\left( R_n^{(1)}(\omega_j)\eta(\omega_j)\right)\right|\leq\mathfrak C \frac 1 {\sqrt n} \sum_{s=0}^{\infty}\|\Phi_s\| \sum_{t=0}^{\infty} t\|\Phi_t\| \sum_{u=-\infty}^{\infty} \|\mathfrak f_u\|\overset{n\to\infty}{\longrightarrow}0.
	\end{eqnarray*}
~\hfill$\Box$
\end{proof}

This lemma helps to deduce \Cref{Prop2}, which can be seen as the main part of the proof of the asymptotic normality of the Whittle estimator.\\[2mm]

 \noindent\textit{Proof of \Cref{Prop2}.}
	Due to \Cref{Prop1}, we get
	\begin{eqnarray*}
		\lefteqn{\frac{1}{2\sqrt n}\sum_{j=-n+1}^{n}\tr\left(\eta(\omega_j)I_n(\omega_j)- \eta(\omega_j)\f(\omega_j)\right)}\\
		&=&\frac{1}{2\sqrt n}\sum_{j=-n+1}^{n}\left\{\tr\left(I_{n,N}(\omega_j)\Phi(e^{i\omega_j})^\top\eta(\omega_j)\Phi(e^{-i\omega_j})\right)-\tr\left( \eta(\omega_j)\f(\omega_j)\right)\right\}+o_\P(1).
	\end{eqnarray*}
	We define$$q(\omega):= \Phi(e^{i\omega})^\top\eta(\omega)\Phi(e^{-i\omega}), \quad \omega \in [-\pi,\pi],$$ and approximate $q$ by its Fourier series of degree $M$, namely,
	\begin{align}\label{Refb_0}
	q_M(\omega)=\sum_{|k|\leq M}b_k e^{ik\omega} \quad \text{ where } \quad b_k=\frac{1}{2\pi}\int_{-\pi}^{\pi}e^{-ik\omega} q(\omega) d\omega,\quad k \in \Z.
	\end{align}
	The coefficients $b_k$ satisfy \begin{eqnarray}
	\nonumber \sum_{k=-\infty}^{\infty}\|b_k\| |k|^{1/2}&=&\sum_{k=-\infty}^{\infty} \left\| \frac{1}{2\pi} \int_{-\pi}^{\pi}e^{-ik\omega}\Phi(e^{i\omega}) ^\top\eta(\omega)\Phi(e^{-i\omega})d\omega \right\| |k|^{1/2}\\
	\nonumber&=& {\sum_{k=-\infty}^{\infty} \left\| \frac{1}{2\pi} \sum_{j=0}^{\infty} \sum_{\ell=0}^{\infty} \sum_{u=-\infty}^{\infty} \Phi_j^\top \mathfrak{f}_u  \Phi_\ell \int_{-\pi}^{\pi} e^{-i(k-j+u+\ell)\omega}d\omega \right\| |k|^{1/2}}\\
	\nonumber&\leq& { \sum_{j=0}^{\infty} \sum_{\ell=0}^{\infty} \sum_{u=-\infty}^{\infty} \|\Phi_j\|  \|\mathfrak{f}_u\| \|\Phi_\ell\| |j-u-\ell|^{1/2}} \\
	\nonumber&\leq&{ \mathfrak C \sum_{j=0}^{\infty} \|\Phi_j\| \left(\max\{1,|j|\}\right)^{1/2} \sum_{u=-\infty}^{\infty} \|\mathfrak{f}_u\| \left(\max\{1,|u|\}\right)^{1/2} \sum_{\ell=0}^{\infty} \|\Phi_\ell\| \left(\max\{1,|\ell|\}\right)^{1/2}}\\\label{endl1/2}&<& { \infty},
	\end{eqnarray} and therefore $\sum_{k=-\infty}^{\infty} \|b_k\|<\infty$ as well. An application of \Cref{FourierKon} of the Supplementary Material leads to $$q_M(\omega)\overset{M \to \infty}{\longrightarrow}q(\omega) \quad \text{ uniformly in } \omega \in [-\pi,\pi].$$
	\textbf{Step 1:} We show \begin{align}\label{zwisch1}\lim_{M\to \infty} \limsup_{n\to \infty}\ \P\left(\frac{1}{\sqrt{n}}\left| \sum_{j=-n+1}^{n}\tr(I_{n,N}(\omega_j)(q(\omega_j)-q_M(\omega_j))\right| > \epsilon\right)=0 \quad \forall\ \epsilon>0.\end{align}
	Consider \begin{eqnarray}
	\nonumber\lefteqn{\frac{1}{\sqrt n}\sum_{j=-n+1}^{n}\tr\left(I_{n,N}(\omega_j)(q(\omega_j)-q_M(\omega_j))\right)}\\
	&=& \frac{\sqrt n}{\pi} \sum_{|k|>M} \tr\left(\sum_{h=-n+1}^{n-1} \overline \Gamma_{n,N}(h)b_k \left(\frac{1}{2n}\sum_{j=-n+1}^{n}e^{-i(h-k)\omega_j}\right)\right). \label{Step1_1}
	\end{eqnarray}
	We investigate the terms with $h=0$ and $h\neq 0$ separately.
	For $h=0$ and $n>M$ we get \begin{align}\noindent
	\left| \frac{\sqrt n}{\pi}\sum_{|k|>M}\tr\left(\overline \Gamma_{n,N}(0) b_k  \mathbf{1}_{\{\exists z \in \Z\setminus\{0\}\ :\ k=2nz\}}\right)\right|\leq{ \mathfrak{C}\sqrt n\|\overline \Gamma_{n,N}(0)\|\sum_{|k|\geq 2n}\|b_k\|  }
	\label{Step1_2}\overset{n \to \infty}{\longrightarrow}0 \quad \P\text{-a.s.,}
	\end{align}
	since \Cref{Gamma_N} in the Supplementary Material and the continuous mapping theorem imply $\|\overline \Gamma_{n,N}(0)\|\overset{a.s.}{\to} \|\Sigma_N^{(\Delta)}\|$. \\
	
	Now, we investigate the terms with $h\not=0$. The independence of the sequence $(N\D_k)_{k\in \N_0}$
	leads to \begin{align*}\E\left[\overline \Gamma_{n,N}(h)\right]=0 \quad \text{for } h \neq 0 \end{align*} and therefore, \begin{align}\label{Step1_3}
	\E\left[  {\sqrt n}\sum_{|k|>M} \tr\left(\left(\sum_{h=1}^{n-1}\overline \Gamma_{n,N}(h)+\sum_{h=-n+1}^{-1}\overline \Gamma_{n,N}(h) \right)b_k\mathbf{1}_{\{\exists z \in \Z\ :\ h=k+2nz\}}\right)\right]=0.
	\end{align}
	Due to \eqref{Step1_1}-\eqref{Step1_3} and the Tschebycheff inequality, for the proof of \eqref{zwisch1} it is sufficient to show that
    \begin{align}\label{Step1_4}\lim_{M\to \infty}\lim_{n\to \infty}\Var\left( {\sqrt n}\sum_{|k|>M} \tr\left(\left(\sum_{h=1}^{n-1}\overline \Gamma_{n,N}(h)+\sum_{h=-n+1}^{-1}\overline \Gamma_{n,N}(h) \right)b_k\mathbf{1}_{\{\exists z \in \Z\ :\ h=k+2nz\}}\right)\right)=0.\end{align}
	First, property \eqref{vectrace}
	and $\E\left\|\vecc\left(\overline{\Gamma}_{n,N}(h)\right)\vecc\left(\overline{\Gamma}_{n,N}(h)\right)^\top\right\|\leq \frac{\mathfrak{C}}{n}$
	result in
	\begin{eqnarray*}
		\lefteqn{\Var\left(  {\sqrt n}\sum_{|k|>M} \tr\left(\left(\sum_{h=1}^{n-1}\overline \Gamma_{n,N}(h)+\sum_{h=-n+1}^{-1}\overline \Gamma_{n,N}(h) \right)b_k\mathbf{1}_{\{\exists z \in \Z\ :\ h=k+2nz\}}\right)\right)}\\
		&&=\Var\left(  {2\sqrt n}\sum_{h=1}^{n-1}\vecc\left(\sum_{|k|>M}b_k^\top\mathbf{1}_{\{\exists z \in \Z\ :\ h=k+2nz\}}\right)^\top\left(I_N\otimes I_N\right)\vecc\left(\overline{\Gamma}_{n,N}(h)\right)\right)\\
		&&\leq{4n}\sum_{h=1}^{n-1}\left\|\vecc\left(\sum_{|k|>M}b_k^\top\mathbf{1}_{\{\exists z \in \Z\ :\ h=k+2nz\}}\right)\right\|^2\left\|\left(I_N\otimes I_N\right)\right\|^2 \left\|\E\left[\vecc\left(\overline{\Gamma}_{n,N}(h)\right)\vecc\left(\overline{\Gamma}_{n,N}(h)\right)^\top\right]\right\|\\
		&&\leq {\mathfrak{C}} \sum_{h=1}^{n-1}\left\|\sum_{|k|>M}b_k \mathbf{1}_{\{\exists z \in \Z\ :\ h=k+2nz\}}\right\|^2\leq  {\mathfrak{C}} \left(\sum_{|k|>M}\|b_k\|\right)^2\overset{M\to\infty}{\longrightarrow}0.
	\end{eqnarray*}\\
	\textbf{Step 2:} We show
	\begin{eqnarray}\label{step2}
	\lefteqn{\frac{1}{\sqrt n}\sum_{j=-n+1}^{n}\left(\tr\left(I_{n,N}(\omega_j)q_M(\omega_j)\right)- \tr\left(\eta(\omega_j)\f(\omega_j)\right)\right)} \nonumber\\
	&&=\frac{\sqrt n}{\pi}\tr\left(\sum_{h=-M}^{M}\left(\overline \Gamma_{n,N}(h)-\Gamma_N(h)\right)b_h\right)+o(1).
	\end{eqnarray}
	Let $M>n$. Then, due to \Cref{Veit} of the Supplementary Material and Parseval's equality, we receive
	\begin{eqnarray}
	\lefteqn{\nonumber\frac{1}{\sqrt n}\sum_{j=-n+1}^{n}\left(\tr\left(I_{n,N}(\omega_j)q_M(\omega_j)\right)- \tr\left(\eta(\omega_j)\f(\omega_j)\right)\right) }\\
	&=&\nonumber\frac{\sqrt n}{\pi}\tr\left(\sum_{h=-M}^{M}\overline\Gamma_{n,N}(h)b_h\right) -\frac {\sqrt n}{\pi} \int_{-\pi}^{\pi} \tr\left(\eta(\omega)\f(\omega)\right)d\omega\\
	&&\nonumber+\frac{\sqrt n}{\pi} \int_{-\pi}^{\pi} \tr\left(\eta(\omega)\f(\omega)\right)d\omega-\tr\left(\frac{1}{\sqrt n}\sum_{j=-n+1}^{n}\eta(\omega_j)\f(\omega_j)\right)
	\\
	&=&\label{A}\frac{\sqrt n}{\pi}\tr\left(\sum_{h=-M}^{M}\overline\Gamma_{n,N}(h)b_h\right) -\frac {\sqrt n}{\pi} \int_{-\pi}^{\pi} \tr\left(\eta(\omega)\f(\omega)\right)d\omega+o(1).
	\end{eqnarray}
	Taking $\Gamma_N(h)=0_{N\times N}$ for $h \neq 0$ into account, we receive
	\begin{eqnarray}\nonumber\lefteqn{\frac{\sqrt n}{\pi}\tr\left(\sum_{h=-M}^{M}\overline\Gamma_{n,N}(h)b_h\right)-\frac {\sqrt n}{\pi} \int_{-\pi}^{\pi} \tr\left(\eta(\omega)\f(\omega)\right)d\omega} \\
	&	=&\nonumber{\frac{\sqrt n}{\pi}\tr\left(\sum_{h=-M}^{M}\left(\overline \Gamma_{n,N}(h)-\Gamma_N(h)\right)b_h\right)}\\
	&& \label{B}+\frac{\sqrt n}{\pi}\left(\tr(\Gamma_{N}(0)b_0) -\int_{-\pi}^{\pi} \tr\left(\eta(\omega)\f(\omega)\right)d\omega\right).\end{eqnarray}
	Using the representation $\f(\omega)=\frac 1 {2\pi}\Phi(e^{-i\omega})\Sigma_N^{(\Delta)} \Phi(e^{i\omega})^\top$ and $q(\omega)=\Phi(e^{i\omega})^\top \eta(\omega)\Phi(e^{-i\omega})$ for $\omega \in [-\pi,\pi]$, yield \begin{eqnarray}
	\nonumber\lefteqn{\frac{\sqrt n}{\pi}\tr\left(\Gamma_{N}(0)b_0\right)- \tr\left(\int_{-\pi}^{\pi} \eta(\omega)\f(\omega)d\omega\right)}\\ &=&\nonumber{\frac{\sqrt n}{\pi}\int_{-\pi}^{\pi}\tr\left(\frac 1 {2\pi} \Sigma_N^{(\Delta)} q(\omega)\right)- \tr\left(\eta(\omega)\f(\omega)\right)d\omega}\\
	&=&\label{C}{\frac{\sqrt n}{\pi}\int_{-\pi}^{\pi}\tr\left( \eta(\omega)\frac 1 {2\pi}\Phi(e^{-i\omega})\Sigma_N^{(\Delta)} \Phi(e^{i\omega})^\top -\eta(\omega)\f(\omega)\right)d\omega=0.}
	\end{eqnarray}
	Then, \eqref{A}-\eqref{C} result in \eqref{step2}.~\\[2mm]
	\textbf{Step 3:} Next, we prove the asymptotic normality  \begin{align}\label{step3}\frac{\sqrt n}{2\pi}\tr\left(\sum_{h=-M}^{M}\left(\overline \Gamma_{n,N}(h)-\Gamma_N(h)\right)b_h\right)  \overset{\mathcal D}{\longrightarrow}\mathcal N(0,\Sigma_\eta(M)),\end{align} where $\Sigma_\eta(M)$ is defined as $$\Sigma_\eta(M)=\frac{1}{\pi^2}\sum_{h=1}^{M}\tr\left(b_h\Sigma_N^{(\Delta)} b_{h}^H \Sigma_N^{(\Delta)}\right)+\frac1 {4\pi^2}\vecc(b_0^\top)^\top\left(\E\left[N_1^{(\Delta)}N_1^{(\Delta)\top}\otimes N_1^{(\Delta)}N_1^{(\Delta)\top}\right]-\Sigma_N^{(\Delta)}\otimes \Sigma_N^{(\Delta)}\right)\vecc(b_0^H).$$ Therefore, we consider
	\begin{eqnarray}\nonumber
	\lefteqn{\frac{\sqrt n}{2\pi}\tr\left(\sum_{h=-M}^{M}\left(\overline \Gamma_{n,N}(h)-\Gamma_N(h)\right)b_h\right)}\\
	\label{b0Zerlegung}&&=\frac{1}{\pi}\sum_{h=1}^{M}\sqrt n\tr\left(\left(\overline \Gamma_{n,N}(h)-\Gamma_N(h)\right)b_h\right)+\frac{\sqrt n}{2\pi}\tr\left(\left(\overline \Gamma_{n,N}(0)-\Gamma_N(0)\right)b_0\right).\end{eqnarray}
	Writing $$\sqrt n\tr\left(\left(\overline \Gamma_{n,N}(h)-\Gamma_N(h)\right)b_h\right)=\sqrt n \vecc(b_h^\top)^\top \vecc\left( \overline \Gamma_{n,N}(h)-\Gamma_N(h)\right),$$
	an application of \Cref{asyNGamma} of the Supplementary Material leads to
	\begin{align*}
	\sqrt n\tr\left(\left(\overline \Gamma_{n,N}(h)-\Gamma_N(h)\right)b_h\right)\overset{\mathcal D}{\longrightarrow}\mathcal{N}_h,
	\end{align*} where $(\mathcal{N}_h)_{h\in\N_0}$ is an independent centered normally distributed sequence of random vectors with covariance matrix $$ \Sigma_{\mathcal{N}_h}:=\vecc(b_h^\top)^\top\left(\Sigma_N^{(\Delta)} \otimes \Sigma_N^{(\Delta)}\right) \vecc(b_h^H)=\tr\left(b_h\Sigma_N^{(\Delta)} b_{h}^H \Sigma_N^{(\Delta)}\right) \quad \text{for } h \neq 0$$ and
	$$\Sigma_{\mathcal{N}_0}:=\vecc(b_0^\top)^\top\left(\E\left[N_1^{(\Delta)}N_1^{(\Delta)\top}\otimes N_1^{(\Delta)}N_1^{(\Delta)\top}\right]-\Sigma_N^{(\Delta)}\otimes \Sigma_N^{(\Delta)}\right)\vecc(b_0^H) .$$
	Finally, 
	\begin{align*}
	&\frac{\sqrt n}{2\pi}\tr\left(\sum_{h=-M}^{M}\left(\overline \Gamma_{n,N}(h)-\Gamma_N(h)\right)b_h\right)\overset{\mathcal{D}}{\longrightarrow}\mathcal{N}\left(0,\frac{1}{\pi^2}\sum_{h=1}^{M}\tr\left(b_h\Sigma_N^{(\Delta)} b_{h}^H \Sigma_N^{(\Delta)}\right)\right.\\
	&\qquad\qquad \qquad\qquad\quad\ \ +\left.\frac1 {4\pi^2}\vecc(b_0^\top)^\top\left(\E\left[N_1^{(\Delta)}N_1^{(\Delta)\top}\otimes N_1^{(\Delta)}N_1^{(\Delta)\top}\right]-\Sigma_N^{(\Delta)}\otimes \Sigma_N^{(\Delta)}\right)\vecc(b_0^H)\right).\end{align*}
	\textbf{Step 4:} We show \begin{align}
	&\nonumber\frac{1}{\pi^2}\sum_{h=1}^{M}\tr\left(b_h\Sigma_N^{(\Delta)} b_{h}^H \Sigma_N^{(\Delta)}\right)+\frac1 {4\pi^2}\vecc(b_0^\top)^\top\left(\E\left[N_1^{(\Delta)}N_1^{(\Delta)\top}\otimes N_1^{(\Delta)}N_1^{(\Delta)\top}\right]-\Sigma_N^{(\Delta)}\otimes \Sigma_N^{(\Delta)}\right)\vecc(b_0^H)\\
	\overset{M\to \infty}{\longrightarrow}&\nonumber\frac{1}{\pi}\int_{-\pi}^{\pi}\tr\left(\eta(\omega)\f(\omega)\eta(\omega) \f(\omega)\right)d\omega+ \frac{1}{16\pi^4}\int_{-\pi}^{\pi}\vecc\left(\Phi(e^{-i\omega})^\top\eta(\omega)^\top\Phi(e^{i\omega})\right)^\top d\omega\\
	&\label{step4}\quad\quad\left(\E\left[N_1^{(\Delta)}N_1^{(\Delta)\top}\otimes N_1^{(\Delta)}N_1^{(\Delta)\top}\right]-3\Sigma_N^{(\Delta)}\otimes \Sigma_N^{(\Delta)}\right)\int_{-\pi}^{\pi}\vecc\left(\Phi(e^{i\omega})^\top\eta(\omega)\Phi(e^{-i\omega})\right)d\omega.
	\end{align}
	Therefore, note that
	\begin{align*}
	&\frac{1}{\pi^2}\sum_{h=1}^{M}\tr\left(b_h\Sigma_N^{(\Delta)} b_{h}^H \Sigma_N^{(\Delta)}\right)+\frac1 {4\pi^2}\vecc(b_0^\top)^\top\left(\E\left[N_1^{(\Delta)}N_1^{(\Delta)\top}\otimes N_1^{(\Delta)}N_1^{(\Delta)\top}\right]-\Sigma_N^{(\Delta)}\otimes \Sigma_N^{(\Delta)}\right)\vecc(b_0^H)\\
	&\quad\overset{M\to \infty}{\longrightarrow}\frac{1}{\pi^2}\sum_{h=1}^{\infty}\tr\left(b_h\Sigma_N^{(\Delta)} b_{h}^H \Sigma_N^{(\Delta)}\right) +\frac1 {4\pi^2}\vecc(b_0^\top)^\top\left(\E\left[N_1^{(\Delta)}N_1^{(\Delta)\top}\otimes N_1^{(\Delta)}N_1^{(\Delta)\top}\right]-\Sigma_N^{(\Delta)}\otimes \Sigma_N^{(\Delta)}\right)\vecc(b_0^H).
	\end{align*}
	But
	\begin{eqnarray*}
		\frac{1}{2\pi}\sum_{h=-\infty}^{\infty} \tr\left(b_h\Sigma_N^{(\Delta)}b_h^H\Sigma_N^{(\Delta)}\right)&=&\frac{1}{4\pi^2}\sum_{h=-\infty}^{\infty}\sum_{\ell=-\infty}^{\infty}\tr\left(b_h
		\Sigma_N^{(\Delta)}b_\ell^H\Sigma_N^{(\Delta)}\right)\int_{-\pi}^{\pi}e^{i(h-\ell)\omega}d\omega\\
		&=&\frac{1}{4\pi^2}\int_{-\pi}^{\pi}\tr\left(q(\omega)\Sigma_N^{(\Delta)}q(\omega)^H\Sigma_N^{(\Delta)}\right)d\omega\\
		&=&\int_{-\pi}^{\pi}\tr\left(\eta(\omega)\f(\omega)\eta(\omega)^H \f(\omega)\right)d\omega,
	\end{eqnarray*}
	where we plugged in the definition of $q$ in the last equality.
	Eventually, due to the representation of $b_0$, we receive
	\begin{eqnarray*}
		\lefteqn{\frac{1}{\pi^2}\sum_{h=1}^{\infty}\tr\left(b_h\Sigma_N^{(\Delta)} b_{h}^H \Sigma_N^{(\Delta)}\right)+\frac1 {4\pi^2}\vecc(b_0^\top)^\top\left(\E\left[N_1^{(\Delta)}N_1^{(\Delta)\top}\otimes N_1^{(\Delta)}N_1^{(\Delta)\top}\right]-\Sigma_N^{(\Delta)}\otimes \Sigma_N^{(\Delta)}\right)\vecc(b_0^H)}\\
		&=&\frac{1}{\pi}\int_{-\pi}^{\pi}\tr\left(\eta(\omega)\f(\omega)\eta(\omega) \f(\omega)\right)d\omega+ \frac{1}{16\pi^4}\int_{-\pi}^{\pi}\vecc\left(\Phi(e^{-i\omega})^\top\eta(\omega)^\top\Phi(e^{i\omega})\right)^\top d\omega\\
		&&\left(\E\left[N_1^{(\Delta)}N_1^{(\Delta)\top}\otimes N_1^{(\Delta)}N_1^{(\Delta)\top}\right]-3\Sigma_N^{(\Delta)}\otimes \Sigma_N^{(\Delta)}\right)\int_{-\pi}^{\pi}\vecc\left(\Phi(e^{i\omega})^\top\eta(\omega)\Phi(e^{-i\omega})\right)d\omega.
	\end{eqnarray*}
	Finally, Step 3, Step 4 and a multivariate version of Problem 6.16 of \cite{BrockwellDavis} give
	$$\frac{\sqrt n}{2\pi}\tr\left(\sum_{h=-M}^{M}\left(\overline \Gamma_{n,N}(h)-\Gamma_N(h)\right)b_h\right)
	\overset{\mathcal D,n\to\infty}{\longrightarrow}\mathcal N(0,\Sigma_\eta(M))
	\overset{\mathcal D,M\to\infty}{\longrightarrow}\mathcal N(0,\Sigma_\eta).$$  Along with Step 1, Step 2 and  Proposition 6.3.9 of \cite{BrockwellDavis}, the statement follows.
\hfill$\Box$\\[1mm]

\noindent \textit{Proof of \Cref{abl1W}.}
	The proof is based on the Cram\'er Wold Theorem and \Cref{Prop2}. Therefore, let $\lambda=(\lambda_1,\ldots,\lambda_r)^\top \in \R^r$.
	We obtain
	\begin{align*}
	\sqrt{n}\left[\nabla_\vartheta W_n(\vartheta_0)\right]\lambda=&\frac 1{2\sqrt n}\sum_{j=-n+1}^{n}\nabla_\vartheta\left[\tr\left(\f(\omega_j,\vartheta_0)^{-1} I_n(\omega_j)\right)+\log(\det(\f(\omega_j,\vartheta_0)))\right]\lambda\\
	=&\frac 1{2\sqrt n}\sum_{j=-n+1}^{n}\left[\sum_{t=1}^{r}\tr\left(-\lambda_t\f(\omega_j)^{-1}\left( \frac{\partial}{\partial \vartheta_t}\f(\omega_j,\vartheta_0)\right)\f(\omega_j)^{-1} I_n(\omega_j)\right)\right]\\
	&+\frac{1}{2\sqrt n}\sum_{j=-n+1}^{n}\nabla_\vartheta[\tr(\log(\f(\omega_j,\vartheta_0)))]\lambda.\end{align*}
	We define the matrix function $\eta_\lambda:[-\pi,\pi]\to \C^{m\times m}$ as \begin{eqnarray}\label{eta} \eta_\lambda(\omega)=-\sum_{t=1}^{r}\lambda_t\f(\omega)^{-1}\left( \frac{\partial}{\partial \vartheta_t}\f(\omega,\vartheta_0)\right)\f(\omega)^{-1},  \quad    \omega \in [-\pi,\pi].\end{eqnarray}
	Furthermore, \
	\begin{align*} &\tr\left(\frac {\partial}{\partial \vartheta_t } \log\left(\f(\omega,\vartheta_0)\right)\right)=\tr\left(\f(\omega)^{-1}\left(\frac {\partial}{\partial \vartheta_t}\f(\omega,\vartheta_0)\right)\right).\end{align*}
	Then,
	\begin{align*}
	\sqrt n \left[\nabla_\vartheta W_n(\vartheta_0)\right]\lambda
	=&\frac{1}{2\sqrt n}\sum_{j=-n+1}^{n}\tr\left(\eta_\lambda(\omega_j)\left(I_n(\omega_j)-\f(\omega_j)\right)\right).
	\end{align*}
	Apparently, $\eta_\lambda$ is two times continuously differentiable by \Cref{BemAsyNorm} and $2\pi$ periodic. Moreover, every component of the Fourier coefficients $(\mathfrak f_{\lambda,u})_{u \in \Z}$ of $\eta_\lambda$ satisfies $\sum_{u=-\infty}^{\infty} |\mathfrak f_{\lambda,u}[k,\ell]| |u|^{1/2} <\infty, k,\ell \in \{1,\ldots, m\}$ (see \cite{BrockwellDavis}, Exercise 2.22 applied to $\eta_\lambda$ and its derivative $ \eta_\lambda'$), and therefore,\linebreak
	$\sum_{u=-\infty}^{\infty} \|\mathfrak f_{\lambda,u}\| |u|^{1/2} <\infty$ follows. Then, due to \Cref{Prop2}, we get as $n\to\infty$,
	\begin{align*}
	\sqrt{n}\left[\nabla_\vartheta W_n(\vartheta_0)\right]\lambda \overset{\mathcal D}{\longrightarrow}\mathcal N(0,\Sigma_{\lambda^\top\nabla_\vartheta W}),
	\end{align*} where \begin{eqnarray*}\Sigma_{\lambda^\top\nabla_\vartheta W}&=&\frac{1}{\pi}\int_{-\pi}^{\pi}\tr\left(\eta_\lambda(\omega)\f(\omega)\eta_\lambda(\omega) \f(\omega)\right)d\omega
		\\
		&&+\frac{1}{16\pi^4}\int_{-\pi}^{\pi}\vecc\left(\Phi(e^{-i\omega})^\top\eta_\lambda(\omega)^\top\Phi(e^{i\omega})\right)^\top d\omega\\
		&&\qquad\cdot\left(\E\left[N_1^{(\Delta)}N_1^{(\Delta)\top}\otimes N_1^{(\Delta)}N_1^{(\Delta)\top}\right]-3\Sigma_N^{(\Delta)}\otimes \Sigma_N^{(\Delta)}\right)\int_{-\pi}^{\pi}\vecc\left(\Phi(e^{i\omega})^\top\eta_\lambda(\omega)\Phi(e^{-i\omega})\right)d\omega\\
		&=:&\Sigma_{\lambda,1}+\Sigma_{\lambda,2}+\Sigma_{\lambda,3}.\end{eqnarray*}We investigate the three terms separately. With \eqref{vectrace}, the first term fulfills the representation
	\begin{eqnarray*}
		\Sigma_{\lambda,1}
		&=&\frac{1}{\pi}\int_{-\pi}^{\pi}\tr\left(\left(\sum_{t=1}^{r}\lambda_t \frac{\partial}{\partial \vartheta_t}\f(\omega,\vartheta_0)\right)\f(\omega)^{-1}\left(\sum_{s=1}^{r}\lambda_s \frac{\partial}{\partial \vartheta_s}\f(\omega,\vartheta_0)\right)\f(\omega)^{-1}\right)d\omega\\
		&=&\lambda^\top\left[\frac{1}{\pi}\int_{-\pi}^{\pi}\left(\nabla_\vartheta\f(-\omega,\vartheta_0)\right)^\top\left(\f(-\omega)^{-1}\otimes \f(\omega)^{-1}\right)\nabla_\vartheta\f(\omega,\vartheta_0)\right]\lambda.
	\end{eqnarray*}
	Similarly, we get the representation \begin{eqnarray*}
		\lefteqn{\Sigma_{\lambda,2}=
			\frac{\lambda^\top}{16\pi^4}\left[\int_{-\pi}^{\pi}\int_{-\pi}^{\pi}\E\left[\nabla_\vartheta\f(-\omega,\vartheta_0)^\top\left(\f(-\omega)^{-1}\Phi(e^{i\omega})\otimes \f(\omega)^{-1}\Phi(e^{-i\omega})\right)N_1^{(\Delta)}N_1^{(\Delta)\top}\right.\right.}\\
		&&\left.\left.\quad\quad\quad\quad\otimes N_1^{(\Delta)}N_1^{(\Delta)\top}\left(\Phi(e^{-i\tau})^\top\f(-\tau)^{-1}\otimes \Phi(e^{i\tau})^\top\f(\tau)^{-1}\right)\nabla_\vartheta\f(\tau,\vartheta_0)\right]d\omega d\tau\right]\lambda
	\end{eqnarray*}
	for the second term, and analogously
	\begin{align*}
	&\Sigma_{\lambda,3}=-\frac{3\lambda^\top}{16\pi^4}\left[
	\int_{-\pi}^{\pi}\int_{-\pi}^{\pi}\nabla_\vartheta\f(-\omega,\vartheta_0)^\top\left(\f(-\omega)^{-1}\Phi(e^{i\omega})\Sigma_N^{(\Delta)}\Phi(e^{-i\tau})^\top\f(-\tau)^{-1}\right)\right.\\
	&\left.\quad\quad\quad\quad\quad\quad\quad\quad\otimes\left(\f(\omega)^{-1}\Phi(e^{-i\omega})\Sigma_N^{(\Delta)}\Phi(e^{i\tau})^\top\f(\tau)^{-1}\right)\nabla_\vartheta\f(\tau,\vartheta_0)d\omega d\tau \right]\lambda
	\end{align*}
	for the third term.\hfill$\Box$~\\

\noindent\textit{Proof of \Cref{SatzAsyN}.}
Since  $\widehat\vartheta^{(\Delta)}_n \overset{a.s.}{\longrightarrow}\vartheta_0$ (see \Cref{satz1}) and $\Sigma_{\nabla^2W}$ is positive definite (see  \Cref{positiveSigmanabla2})
the conclusion follows from \eqref{taylor1}, \Cref{remarkunif} and \Cref{abl1W}.
\hfill \ensuremath{\Box}~\\[1mm]

\noindent\textit{Sketch of the proof of \Cref{Remark 3.6}} \label{Subsection:Remark:3.6}
	Let $\Phi_Z$ be the polynomial of the (existing) VAR$(\infty)$ of the \linebreak VARMA$(p,q)$ process.
	\Cref{Prop2}  can be formulated for VARMA processes.
	As in the proof of \Cref{SatzAsyN} we have to plug in there for $\eta$ the function $\eta_\lambda$ as given in \eqref{eta}.
	Then, $b_0$ in \eqref{Refb_0}  has for the VARMA process $(Z_n)_{n\in\N}$ the form
	\begin{eqnarray*}
		b_0	&=&\int_{-\pi}^{\pi}-2\pi \sum_{t=1}^{r}\lambda_t \Sigma_e^{-1}\Phi_Z(e^{-i\omega})^{-1}\left( \frac{\partial}{\partial \vartheta_t}f_Z(\omega,\vartheta_0)\right) \Phi_Z(e^{i\omega})^{\top -1}\Sigma_e^{-1} d\omega\\
		&=&-\Sigma_e^{-1}\int_{-\pi}^{\pi}\sum_{t=1}^{r}\lambda_t \frac{\partial}{\partial \vartheta_t}\log\left(\Phi_Z(e^{-i\omega},\vartheta_0)\right) d\omega-\int_{-\pi}^{\pi} \left(\sum_{t=1}^{r}\lambda_t \frac{\partial}{\partial \vartheta_t}\log\left(\Phi_Z(e^{i\omega},\vartheta_0)\right)\right)^{\top} d\omega\ \Sigma_e^{-1}.
	\end{eqnarray*}
	If $\Phi_Z$ is two times differentiable, the Leibniz rule yield
	\begin{align*}	
	b_0=-\Sigma_e^{-1}\sum_{t=1}^{r}\lambda_t \frac{\partial}{\partial \vartheta_t}\int_{-\pi}^{\pi}\log\left(\Phi_Z(e^{-i\omega},\vartheta_0)\right) d\omega- \left[\sum_{t=1}^{r}\lambda_t \frac{\partial}{\partial \vartheta_t}\int_{-\pi}^{\pi}\log\left(\Phi_Z(e^{i\omega},\vartheta_0)\right) d\omega\right]^{\top} \Sigma_e^{-1}.
	\end{align*}
	Similarly to the proof of Theorem 5.8.1 of \cite{BrockwellDavis}, one can show that the integrals are constant and therefore, that $b_0=0$. For a more detailed approach, we refer to \cite{DunsmuirHannan76}.
\hfill$\Box$

~\hfill$\Box$
\end{proof}
The proof of \Cref{normality2} now matches the proof of \Cref{SatzAsyN}, where \Cref{abl1W} is replaced by \Cref{abl1Wadj} and \Cref{remarkunif} is replaced by \Cref{paadjW}.

\section{Auxiliary Results}
\subsection{Fourier Analysis}
Since all the previous sections make use of Fourier analysis, we state the required basic results. The first property also gives a motivation why the Whittle estimator is based on the frequencies $\{-\frac{\pi(n-1)}{n},\ldots, {\pi}\}$.
\begin{lemma}\label{summee}
	Let $h \in \Z$. Then, $$\frac 1 {2n}\sum_{j=-n+1}^{n}e^{-ih\omega_j}=\Ind_{\{\exists z\in \Z:\ h=2zn\}}.$$
\end{lemma}

We now introduce results which show that an appropriate approximation of the Fourier series exhibit useful convergence properties.
\begin{lemma} \label{FourierKon}
	Let $g:[-\pi,\pi]\to \C$ be continuous. Define $$b_k:= \frac 1 {2\pi}\int_{-\pi}^{\pi} g(\omega)e^{-ik\omega}d\omega\quad\text{ and }\quad q_M(\omega)=\sum_{|k|\leq M}b_ke^{ik\omega}. $$ Suppose that $\sum_{|k|\leq n}|b_k|$ converges. Then $$\sup_{\omega\in [-\pi,\pi]}|q_M(\omega)-g(\omega)|\overset{M \to \infty}{\longrightarrow}0. $$
\end{lemma}
\begin{proof}
	\cite{korner1989fourier}, Theorem 3.1.~\hfill$\Box$
\end{proof}
The assumptions of the previous result are quite strong. If we replace the truncated Fourier series by its Ces\`aro sum, we receive an approximation which exhibits uniform convergence without assuming that the Fourier coefficients are absolute summable. This result is known as Fej\'ers Theorem. Since we want to approximate a parametrized function, we have to adjust Fej\'ers Theorem to a setting which allows a dependency on a second parameter.

\begin{lemma}\label{Fejer}
	Let $\Theta$ be a compact parameter space and $g$ be a continuous real valued function on $[-\pi,\pi]\times \Theta$. Then, the Fourier series of $g$ in the first component is Ces\`aro summable. Further, define the Fourier coefficients $b_k(\vartheta):= \frac 1 {2\pi}\int_{-\pi}^{\pi} g(\omega,\vartheta)e^{ik\omega}d\omega$ and \begin{align*}q_M(\omega,\vartheta)&=\frac{1}{M}\sum_{j=0}^{M-1}\left(\sum_{|k|\leq j}{b}_k(\vartheta) e^{-ik\omega} \right)=\sum_{|k|<M}\left(1-\frac{|k|}{M}\right)b_k(\vartheta)e^{-ik\omega}. \end{align*}
	Then, $$\lim_{M\to \infty} \sup_{\omega \in [-\pi,\pi]}\sup_{\vartheta \in \Theta} |q_M(\omega,\vartheta)-g(\omega,\vartheta)|=0.$$
\end{lemma}
\begin{proof}
	The proof is similar to the proof of Theorem 2.11.1 of \cite{BrockwellDavis} and therefore skipped.~\hfill$\Box$
\end{proof}

\begin{remark}
	If we investigate the Ces\`aro sum of a Fourier series of a matrix valued continuous function $g: [-\pi,\pi]\times \Theta \to \R^{N\times N}$ defined by  $$q_M(\omega):=\frac{1}{M}\sum_{j=0}^{M-1}\left(\sum_{|k|\leq j}{b}_k e^{-ik\omega} \right) \quad \text{ where } \quad b_k:=\frac{1}{2\pi}\int_{-\pi}^{\pi} g(\omega) e^{-ik\omega}d\omega,$$
	Fej\' ers Theorem gives the uniform convergence of each component of $q_M$ to $g$ on $[-\pi,\pi]\times \Theta$. Since $g$ consists of finitely many components, $q_M$ also converges to $g$ uniformly. Obviously, the same holds true for any matrix valued continuous  function  $g: \R\times \Theta \to \R^{N\times N}$ which is $2\pi$ periodic in its first component.
	Similarly, we can transfer \Cref{FourierKon} to matrix valued functions.
\end{remark}
\subsection{The behavior of the sample autocovariance}

We state and prove results concerning the asymptotic behavior of the estimators of the various arising covariance matrices. 
\begin{lemma}\label{ewertmpAC}
	Define the empirical sample autocovariance function
	$$\overline\Gamma_n^{(\Delta)}(h)=\frac1n\sum_{k=1}^{n-h}Y^{(\Delta)}_{k+h}Y^{(\Delta)\top}_{k} \quad \text{ and } \quad  \overline\Gamma_n^{(\Delta)}(-h)=\overline\Gamma_n^{(\Delta)}(h)^\top,  \quad 0\leq h\leq n.$$ Suppose $(A2)$ and $(A3)$ hold. Then,
	for $h\in\Z$ and $n\to\infty$,
	\begin{align*}
	\overline\Gamma_n^{(\Delta)}(h)\overset{a.s.}{\longrightarrow}\Gamma^\Delta(h)
	\end{align*}
	and $\sum_{h=-\infty}^{\infty}\|\Gamma^{(\Delta)}(h)\|<\infty$.
\end{lemma}
\begin{proof}
	Due to Proposition 3.34 of \cite{Marquardt:Stelzer:2007} the process $Y$ is ergodic. Therefore,  Theorem 4.3 of \cite{Krengel} implies that the sampled process $Y^{(\Delta)}$ is ergodic as well. Moreover, $\Gamma_Y(h)=C^\top e^{Ah}\Sigma_N^{(\Delta)}C$ due to \cite{Marquardt:Stelzer:2007}.
	Since the eigenvalues of $A$ have strictly negative real parts $$\sum_{h\in \Z}\|\Gamma^\Delta(h)\|=\sum_{h\in \Z}\|\Gamma_Y(\Delta h )\|< \infty.$$
	Birkhoff`s Ergodic Theorem now leads to $$\overline\Gamma_n^{(\Delta)}(h)\overset{a.s.}{\longrightarrow}\E\left[Y^{(\Delta)}_{h}Y^{(\Delta)\top}_{0}\right]=\Gamma^\Delta(h).$$
~\hfill$\Box$
\end{proof}

\begin{remark}\label{Gamma_N}
	Similarly, one can show that in the situation of \Cref{ewertmpAC} the sample autocovariance function of $N\D$ as introduced in \Cref{invertible} behaves in the same way, i.e. $$\overline \Gamma_{n,N}(h)\overset{a.s.}{\longrightarrow} \Gamma_{N}(h) \quad \forall\ h \in \Z.$$
	Obviously, $\Gamma_N(h)=0$ for $h\neq 0$ and $\Gamma_N(0)=\Sigma_N^{(\Delta)}$.
\end{remark}
Under the stronger assumption of an i.i.d. white noise, the sample autocovariance function has an asymptotic normal distribution.
\begin{lemma}\label{asyNGamma}
	Let $(Z_k)_{k\in \N}$ be an $N$-dimensional i.i.d. white noise with $\E\|Z_1\|^4 <\infty$ and covariance matrix $\Sigma_Z$. Define $$\overline \Gamma_{n,Z}(h)=\frac{1}{n}\sum_{j=1}^{n-h} Z_{j+h}Z_{j}^\top,\quad n\geq h\geq 0,$$ Then, for fixed $\ell \in \N$,
	$$\sqrt n \left(\left[ \begin{array}{c}\vecc\left(\overline \Gamma_{n,Z}(0)\right)\\ \vecc\left(\overline \Gamma_{n,Z}(1)\right)\\ \vdots \\\vecc\left( \overline \Gamma_{n,Z}(\ell)\right)\end{array} \right]-\left[ \begin{array}{c}\vecc\left(\Sigma_{Z}\right)\\ 0\\\vdots \\  0 \end{array} \right]\right)\overset{\mathcal D}{\longrightarrow} \mathcal{N}(0,\Sigma_{\Gamma_Z}(\ell)),$$ where
	$$\Sigma_{\Gamma_Z}(\ell)=\left( \begin{array}{c| c c c}\E[Z_1Z_1^\top \otimes Z_1^\top Z_1]-\Sigma_Z\otimes \Sigma_Z & &0_{N^2\times\ell N^2}& \\\hline &&&\\0_{\ell N^2\times N^2}& & I_\ell \otimes\Sigma_Z\otimes \Sigma_Z& \\&&& \end{array} \right).$$
\end{lemma}
\begin{proof}
	The proof is similar to the proof of Proposition 4.4 in \cite{Lutke} and is therefore omitted.~\hfill$\Box$
\end{proof}
\subsection{Convergence rate of the integral approximation}
To prove the uniform convergence of the Whittle function, it is necessary to guarantee that the deterministic part of the Whittle function converges uniformly.
\begin{lemma}
	\label{unifdet}
	Let $\Theta$ be a compact parameter space and let $g:[-\pi,\pi]\times \Theta \to \C$ be differentiable in the first component. Assume further that $\frac{\partial}{\partial \omega} g(\omega,\vartheta)$ is continuous on $ [-\pi,\pi]\times \Theta$. Then,
	$$\sup_{\vartheta \in \Theta}\left|\frac{1}{2n}\sum_{j=-n+1}^{n}g(\omega_j,\vartheta)-\frac 1 {2\pi}\int_{-\pi}^{\pi}g(\omega,\vartheta)d\omega\right|\overset{n\to \infty}{\longrightarrow}0.$$
\end{lemma}
\begin{proof}
	Follows by an application of the mean value theorem.~\hfill$\Box$
\end{proof}

\begin{lemma}\label{Veit}
	Let $g:[-\pi,\pi] \to \C$ be continuously differentiable. Then,
	$$ \frac{1}{ \sqrt n}\sum_{j=-n+1}^{n}g(\omega_j)-\frac{\sqrt n}{\pi} \int_{-\pi}^{\pi} g(\omega)d\omega\overset{n\to \infty}{\longrightarrow}0$$ holds.
\end{lemma}
\begin{proof}
	The lemma is a consequence of the definition of the Riemann integral and the continuously differentiability of $g$.~\hfill$\Box$
\end{proof}

\section{Extended simulation study} \label{se:further simulation}

In addition to the simulation study of \Cref{sec:simulation}, we investigate bivariate MCAR(1) processes and  CAR(3) processes for both the Brownian motion and the NIG 
driven setting. The parametrization of the MCAR(1) model is given in  Table~1 of \cite{QMLE}
and it is
$$A(\vartheta)=\left( \begin{array}{ c c}\vartheta_1 & \vartheta_2 \\ \vartheta_3 & \vartheta_4 \end{array}\right) = B(\vartheta),\quad C(\vartheta)=\left( \begin{array}{ c c}1 & 0 \\ 0 & 1 \end{array}\right), \quad \Sigma_{L}(\vartheta)=\left( \begin{array}{ c c}\vartheta_5 & \vartheta_6 \\ \vartheta_6 & \vartheta_7 \end{array}\right)$$
in which we choose the parameter $$\vartheta_0^{(3)}=(1,-2,3,-4,0.7513,-0.3536, 0.3536).$$
The results of this simulation study are summarized in \Cref{table_5} and \Cref{table_6}, respectively. Likewise, as for the MCARMA(2,1) model in ~\Cref{table_1} and \Cref{table_2} of \Cref{sec:simulation},
the Whittle estimator and the QMLE  converge very fast.
Furthermore, we use the parameter
\begin{align*}
\vartheta^{(4)}_0=(-0.01,0,7,-1,0.7513,-0.3536,0.3536)
\end{align*}
in this model class. One eigenvalue of $A(\vartheta^{(4)}_0)$ is close to zero. An eigenvalue equal to zero results in a
non-stationary MCARMA process.
\Cref{table_7} shows the results for this setting for $n_2=2000$, and both the Brownian and the NIG driven model.
The Whittle estimator and the QMLE estimate the parameters very well. But it is striking that the bias of several
parameters of the QMLE even vanish.

\begin{table}	
	\begin{center}
		\begin{tabular}{|c||c|c|c||c|c|c|}\hline
			\multicolumn{7}{|c|}{$n_1=500$} \\\hline
			&  \multicolumn{3}{|c||}{Whittle} & \multicolumn{3}{|c|}{QMLE}\\ \hline
			\hspace*{0.1cm} $\vartheta_0$\hspace*{0.1cm} &mean  & bias & std.& mean  & bias &std. \\ \hline
			1&1.0018&0.0018&0.0301& 1.0045& 0.0045&0.0362 \\
			-2&-2.0063&0.0063&0.0321& -2.0068&0.0068& 0.0357 \\
			3&2.9966& 0.0034& 0.0399& 3.0055& 0.0055& 0.0604\\
			-4&-3.9980&0.0020& 0.0399& -4.0019& 0.0019& 0.0565\\ \hline
			0.7513 &0.7543& 0.0030&0.0516 &0.7522& 0.0009& 0.0923\\
			-0.3536 &-0.3573& 0.0037& 0.0463& -0.3531& 0.0005& 0.0674\\
			0.3536 &0.3685& 0.0149& 0.0510& 0.3704& 0.0168& 0.0714\\ \hline \hline
			\multicolumn{7}{|c|}{$n_2=2000$} \\\hline
			&  \multicolumn{3}{|c||}{Whittle} & \multicolumn{3}{|c|}{QMLE}\\ \hline
			\hspace*{0.1cm} $\vartheta_0$\hspace*{0.1cm} &mean  & bias & std.& mean  & bias &std. \\ \hline
			1& 1.0035& 0.0035& 0.0150& 1.0039& 0.0039& 0.0181\\
			-2& -2.0067& 0.0067& 0.0165& -2.0066& 0.0066& 0.0192\\
			3& 2.9991& 0.0009& 0.0192& 3.0021& 0.0021& 0.0286\\
			-4& -3.9987& 0.0013& 0.0223& -4.0003& 0.0003& 0.0302\\ \hline
			0.7513 &0.7532& 0.0019& 0.0257& 0.7514& 0.0001& 0.0401 \\
			-0.3536 & -0.3603& 0.0067&0.0248&-0.3574& 0.0038& 0.0352 \\
			0.3536 & 0.3675& 0.0139& 0.0280&0.3706& 0.0170& 0.0376\\ \hline\hline
			\multicolumn{7}{|c|}{$n_3=5000$} \\\hline
			&  \multicolumn{3}{|c||}{Whittle} & \multicolumn{3}{|c|}{QMLE}\\ \hline
			\hspace*{0.1cm} $\vartheta_0$\hspace*{0.1cm} &mean  & bias & std.& mean  & bias &std. \\ \hline	
			1& 1.0042& 0.0042& 0.0101& 1.0050& 0.0050& 0.0117\\
			-2& -2.0062& 0.0062& 0.0106& -2.0074& 0.0074& 0.0111\\
			3& -2.9996& 0.0004& 0.0114& 3.0021& 0.0021& 0.0169\\
			-4&-3.9965& 0.0035& 0.0158& -4.0013&0.0013& 0.0196 \\ \hline
			0.7513 &0.7537& 0.0024& 0.0173 &0.7549& 0.0036& 0.0258 \\
			-0.3536 &-0.3596& 0.0060& 0.0166& -0.3559& 0.0023& 0.0201  \\
			0.3536 & 0.3663& 0.0027& 0.0169& 0.3693& 0.0157&0.0200\\ \hline 		
		\end{tabular}
	\end{center}
	\caption{\label{table_5} 	Estimation results for a  Brownian motion driven bivariate MCAR(1) process with parameter $\vartheta_0^{(3)}$. }
\end{table}

\begin{table}
	\begin{center}
		\begin{tabular}{|c||c|c|c||c|c|c|}\hline
			\multicolumn{7}{|c|}{$n_1=500$} \\\hline
			&  \multicolumn{3}{|c||}{Whittle} & \multicolumn{3}{|c|}{QMLE}\\ \hline
			\hspace*{0.1cm} $\vartheta_0$\hspace*{0.1cm} &mean  & bias & std. & mean  & bias &std.  \\ \hline
			1&0.9905& 0.0095& 0.0407&0.9806& 0.0194& 0.0460\\
			-2& -1.9871& 0.0129& 0.0531& -2.0038& 0.0038& 0.0579\\
			3&2.9920& 0.0080& 0.0579& 2.9240& 0.0760& 0.0842\\
			-4&-3.9409& 0.0591& 0.1027& -3.9918& 0.0082& 0.0894\\ \hline
			0.7513 &0.7281& 0.0232& 0.1869& 0.7125& 0.0388& 0.0568 \\
			-0.3536 & -0.3366& 0.0170& 0.0302& -0.3251& 0.0285& 0.0497  \\
			0.3536 &0.3381& 0.0155& 0.0335& 0.3182& 0.0354& 0.0486\\ \hline \hline
			\multicolumn{7}{|c|}{$n_2=2000$} \\\hline
			&  \multicolumn{3}{|c||}{Whittle} & \multicolumn{3}{|c|}{QMLE}\\ \hline
			\hspace*{0.1cm} $\vartheta_0$\hspace*{0.1cm} &mean  & bias & std. & mean  & bias &std.  \\ \hline
			1&0.9916& 0.0084& 0.0261& 0.9839& 0.0161& 0.0316\\
			-2&-1.9892& 0.0110& 0.0321&-2.0072& 0.0072& 0.0320 \\
			3&2.9797& 0.0203& 0.0416& 2.9377& 0.0623& 0.0576\\
			-4&-3.9700& 0.0300& 0.0767& -4.0051& 0.0051& 0.0561\\ \hline
			0.7513 &0.7489& 0.0024& 0.1392& 0.7210& 0.0303& 0.0351 \\
			-0.3536 & -0.3603& 0.0067& 0.0241& -0.3224& 0.0312& 0.0312  \\
			0.3536 & 0.3417& 0.0119& 0.0224& 0.3352& 0.0184& 0.0300 \\ \hline \hline
			\multicolumn{7}{|c|}{$n_3=5000$} \\\hline
			&  \multicolumn{3}{|c||}{Whittle} & \multicolumn{3}{|c|}{QMLE}\\ \hline
			\hspace*{0.1cm} $\vartheta_0$\hspace*{0.1cm} &mean  & bias & std. & mean  & bias &std.  \\ \hline
			1& 0.9952& 0.0048& 0.0186& 0.9810& 0.0190& 0.0240\\
			-2& -1.9890& 0.0110& 0.0253& -2.0086& 0.0086& 0.0289\\
			3& 2.9789& 0.0211& 0.0365& 2.9341& 0.0659& 0.0478\\
			-4&-3.9849& 0.0151& 0.0611&-4.0064& 0.0064& 0.0516\\ \hline
			0.7513 &0.7500& 0.0013& 0.0749& 0.6912& 0.0601& 0.0428 \\
			-0.3536 & -0.3600& 0.0064& 0.0148& -0.3412& 0.0124& 0.0237  \\
			0.3536 &0.3499& 0.0037& 0.0201& 0.3208& 0.0328& 0.0238 \\ \hline 			
		\end{tabular}
	\end{center}
	\caption{\label{table_6} 	Estimation results for a  NIG driven bivariate MCAR(1) process  with parameter $\vartheta_0^{(3)}$.}
\end{table}

\begin{table}	
	\begin{center}
		\begin{tabular}{|c||c|c|c||c|c|c|}\hline
			\multicolumn{7}{|c|}{Brownian motion driven,\ $n_2=2000$} \\\hline
			&  \multicolumn{3}{|c||}{Whittle} & \multicolumn{3}{|c|}{QMLE}\\ \hline
			\hspace*{0.1cm} $\vartheta_0$\hspace*{0.1cm} &mean & bias& std.& mean & bias &std.  \\ \hline
			-0.01 &-0.0099& 0.0001&0.0005& -0.0103&0.0003&0 \\
			0 &0&0&0&0&0&0.1891 \\
			7 &6.9245&0.0755&0.0853&7&0&0.0012 \\
			-1 &-1.0442&0.0442&0.1915&-1&0&0.0019\\ \hline
			0.7513 &0.8574&0.1061&0.2193&0.7513&0&0.0031  \\
			-0.3536 &-0.3492&0.0044&0.0587&-0.3535&0.0001&0.0013  \\
			0.3536 &0.7958&0.4422&0.4160&0.3536 &0&0.0005 \\ \hline\hline
			\multicolumn{7}{|c|}{NIG driven,\ $n_2=2000$} \\\hline
			&  \multicolumn{3}{|c||}{Whittle} & \multicolumn{3}{|c|}{QMLE}\\ \hline
			\hspace*{0.1cm} $\vartheta_0$\hspace*{0.1cm} &mean  & bias& std.& mean  & bias &std.  \\ \hline
			-0.01 &-0.0125& 0.0025& 0.0534& -0.099& 0.0001&0.0001 \\
			0 &-0.0084& 0.0084& 0.0507& 0& 0& 0.1805 \\
			7 &7.0137& 0.0137& 0.1081& 7& 0&0.0180 \\
			-1 & -0.8731& 0.1269&0.1354& -1 & 0&0.0049\\ \hline
			0.7513 &1.4557&0.7045&0.0959 &0.7513&0&0.0027  \\
			-0.3536 &0.1189& 0.4724& 0.1675&-0.3536&0&0.0017 \\
			0.3536 &0.7397& 0.3862&0.0524& 0.3535&0.0001&0.0008 \\ \hline
		\end{tabular}
	\end{center}
	\label{table5}
	\begin{center}
	 \caption{ \label{table_7} 	Estimation results for a bivariate MCAR(1) process with parameter $\vartheta_0^{(4)}$ close to the non-stationary case. }
	\end{center}
\end{table}

For the univariate CAR(3) processes with parametrization
\begin{align*}
A(\vartheta)&=\left( \begin{array}{c c c}0 & 1 &0 \\ 0&0&1 \\ \vartheta_1 & \vartheta_2 &\vartheta_3\end{array}\right),\quad \quad B(\vartheta)=\left( \begin{array}{c}0\\0\\\vartheta_1 \end{array}\right),\quad C(\vartheta)=(1 \ 0 \ 0).
\end{align*}
and $$\vartheta_0^{(5)}=(-6, -11, -6),$$
we choose once again the Brownian motion and the NIG Lévy process as driving processes.
The results are documented in \Cref{table_8} and \Cref{table_9}. They correspond to the results
of ~\Cref{table_3} and \Cref{table_4}, respectively for CARMA(2,1) processes.

\begin{table}	
	\begin{center}
		\begin{tabular}{|c||c|c|c||c|c|c||c|c|c|}\hline
			\multicolumn{10}{|c|}{$n_1=500$} \\\hline
			&  \multicolumn{3}{|c||}{Whittle} &\multicolumn{3}{|c||}{adjusted Whittle}& \multicolumn{3}{|c|}{QMLE}\\ \hline
			\ $\vartheta_0$ \ &mean& bias& std.& mean & bias&std.& mean & bias&std.\\ \hline
			-6 &-5.9230&0.0770&0.2074&-6.2266&0.2266&0.6347&-6.4357&0.4357&1.3266\\
			-11 &-10.8390&0.1610&0.4119&-11.2759&0.2759&0.9351& -11.6067& 0.6067& 1.6706\\
			-6 &-5.8267&0.1733&0.3585&-6.0575&0.0575&0.4800&-6.3039&0.3039&1.2821 \\  \hline\hline
			\multicolumn{10}{|c|}{$n_2=2000$} \\\hline			&  \multicolumn{3}{|c||}{Whittle} &\multicolumn{3}{|c||}{adjusted Whittle}& \multicolumn{3}{|c|}{QMLE}\\ \hline
			\ $\vartheta_0$ \ &mean & bias& std.& mean  & bias&std. &mean& bias&std. \\ \hline
			-6 & -5.9886&0.0114&0.1117&-6.0410&0.0410&0.2391&-6.0549&0.0549& 0.4510\\
			-11 & -10.9336&0.0664&0.2372&-11.0680&0.0680&0.4126&-11.0422& 0.0422&0.6005\\
			-6 & -5.8855& 0.1145&0.1755&-5.9460&0.0540&0.1924& -5.9542&0.0458& 0.4464\\ \hline\hline
			\multicolumn{10}{|c|}{$n_3=5000$} \\\hline			&  \multicolumn{3}{|c||}{Whittle} &\multicolumn{3}{|c||}{adjusted Whittle}& \multicolumn{3}{|c|}{QMLE}\\ \hline
			\ $\vartheta_0$\  &mean & bias& std.&mean & bias&std.& mean & bias &std.\\ \hline	
			-6 & -5.9856&0.0144&0.0884&-6.0455&0.0455&0.1444& -5.9861& 0.0139& 0.1120\\
			-11 & -10.9335&0.0665&0.1471&-11.0349&0.0349&0.1298& -10.9259& 0.0741& 0.1877\\
			-6 & -5.9123&0.0877&0.1262&-5.9303&0.0697&0.1104& -5.8937& 0.1063& 0.1406 \\ \hline
			
		\end{tabular}
	\end{center}
	\begin{center}
		\caption{\label{table_8} 	Estimation results for a  Brownian motion driven CAR(3) process with $\vartheta_0^{(5)}$.}
	\end{center}
\end{table}

\begin{table}
	\begin{center}
		\begin{tabular}{|c||c|c|c||c|c|c||c|c|c|}\hline
			\multicolumn{10}{|c|}{$n_1=500$} \\\hline
			&  \multicolumn{3}{|c||}{Whittle} &\multicolumn{3}{|c||}{adjusted Whittle}& \multicolumn{3}{|c|}{QMLE}\\ \hline
			\hspace*{0.1cm} $\vartheta_0$\hspace*{0.1cm} &mean & bias& std. & mean   & bias  &std.  &mean & bias&std.  \\ \hline
			-6 & -5.9449&0.0551&0.4322&-5.9238&0.0762&0.4799&-6.8247& 0.8247& 1.9413\\
			-11 &-10.9222&0.0778&0.5765&-10.9049&0.0951&0.6813& -12.1860& 1.1860& 2.3377\\
			-6 &-5.8492&0.1508&0.3455&-5.8000&0.2000&0.4239& -6.6137& 0.6137& 1.6559 \\  \hline\hline
			\multicolumn{10}{|c|}{$n_2=2000$} \\\hline
			&  \multicolumn{3}{|c||}{Whittle} &\multicolumn{3}{|c||}{adjusted Whittle}& \multicolumn{3}{|c|}{QMLE}\\ \hline
			\hspace*{0.1cm} $\vartheta_0$\hspace*{0.1cm} &mean & bias & std.& mean   & bias  &std.  & mean & bias&std. \\ \hline
			-6 & -5.9611&0.0389&0.1287&-6.0737&0.0737&0.3438&-6.01035& 0.1035& 0.6401 \\
			-11 &-10.9011&0.0989&0.2590&-11.0504&0.0504&0.4832 &-11.1053& 0.1053& 0.8271\\
			-6 &-5.8879&0.1121&0.1988&-5.9692&0.0308&0.2175& -6.0036& 0.0036& 0.5522 \\  \hline\hline
			\multicolumn{10}{|c|}{$n_3=5000$} \\\hline
			&  \multicolumn{3}{|c||}{Whittle} &\multicolumn{3}{|c||}{adjusted Whittle}& \multicolumn{3}{|c|}{QMLE}\\ \hline
			\hspace*{0.1cm} $\vartheta_0$\hspace*{0.1cm} &mean  & bias& std. & mean   & bias  &std.  & mean & bias &std.  \\ \hline
			-6 &-6.0313&0.0313&0.0825&-6.0622&0.0622&0.1883& -6.0087& 0.0087& 0.2748 \\
			-11 & -10.8882&0.1118&0.1274&-11.0345&0.0345&0.1490& -10.9541& 0.0459& 0.3830\\
			-6 & -5.9110&0.0190&0.0885&-5.8438&0.1562&0.2144& -5.9164& 0.0836& 0.2513\\  \hline
			
		\end{tabular}
	\end{center}
	\label{table7}
	\begin{center}
		\caption{\label{table_9}	Estimation results for a  NIG driven CAR(3) process with parameter $\vartheta_0^{(5)}$. }
	\end{center}
\end{table}



\end{document}